\documentclass{article}
\usepackage[utf8]{inputenc}
\usepackage{amsmath}
\usepackage{amssymb}
\usepackage{amsthm}
\usepackage{color}
\usepackage{bbm}
\usepackage{hyperref}
\usepackage[margin=0.9in]{geometry}
\usepackage[backend=biber,
style=alphabetic,
]{biblatex}
\newcommand{\EE}{\mathbb{E}}
\newcommand{\RR}{\mathbb{R}}

\newcommand{\Pro}{\mathbb{P}}
\newcommand{\Nbb}{\mathbb{N}}

\newcommand{\LL}{\mathcal{L}}

\newcommand{\CC}{\mathcal{C}}

\newcommand{\Ecal}{\mathcal{E}}

\newcommand{\FF}{\mathcal{F}}

\newcommand{\HH}{\mathcal{H}}
\newcommand{\Hbb}{\mathbb{H}}
\newcommand{\MM}{\mathcal{M}}

\newcommand{\Scal}{\mathcal{S}}
\newcommand{\NN}{\mathcal{N}}

\newcommand{\XX}{\mathcal{X}}
\newcommand{\YY}{\mathcal{Y}}

\newcommand{\Xbb}{\mathbb{X}}

\newcommand{\norm}[1]{\left\Vert #1\right\Vert}

\newcommand{\floor}[1]{\lfloor #1\rfloor}

\renewcommand{\b}[1]{\left\{ #1\right\}}
\renewcommand{\c}[1]{\left[ #1\right]}
\newcommand{\p}[1]{\left( #1\right)}
\newcommand{\abs}[1]{\left| #1\right|}
\newcommand{\iid}{\overset{i.i.d}{\sim}}
\newcommand{\iidt}{\textit{i.i.d }}
\newcommand{\ind}[1]{\mathbbm{1}_{#1}}

\newcommand{\inner}[1]{\left\langle#1\right\rangle}

\DeclareMathOperator{\Span}{span}
\DeclareMathOperator{\Tr}{Tr}

\DeclareMathOperator{\argmin}{argmin}
\DeclareMathOperator{\Diag}{Diag}
\DeclareMathOperator{\supp}{supp}

\newcommand{\rk}{rk}
\newcommand{\vertiii}[1]{{\left\vert\kern-0.25ex\left\vert\kern-0.25ex\left\vert #1 
		\right\vert\kern-0.25ex\right\vert\kern-0.25ex\right\vert}}
\newcommand{\essinf}{ess~inf}
\newcommand{\esssup}{ess~sup}

\newtheorem{theorem}{Theorem}
\newtheorem{lemma}{Lemma}[theorem]
\newtheorem{definition}{Definition}
\newtheorem{proposition}{Proposition}[theorem]
\newtheorem{corollary}{Corollary}[theorem]
\newtheorem{remark}{remark}

\newtheorem{model}{Model}
\newtheorem{assumption}{Assumption}

\title{$L^2-$posterior contraction rates for Gaussian process and random series priors in Bayesian nonparametric regression models}
\author{Paul Rosa \\ University of Cambridge}
\date{\today}

\begin{document}
	
	\maketitle
	
	\begin{abstract}
		The nonparametric regression model with normal errors has been extensively studied, both from the frequentist and Bayesian viewpoint. A central result in Bayesian nonparametrics is that under assumptions on the prior, the data-generating distribution (assuming a true frequentist model) and a semi-metric $\rho(.,.)$ on the space of regression functions that satisfy the so called testing condition, the posterior contracts around the true distribution with respect to $\rho(.,.)$, and the rate of contraction can be estimated. In the regression setting, the semi-metric $\rho(.,.)$ is often taken to be the Hellinger distance or the empirical $L^2$ norm (i.e., the $L^2$ norm with respect to the empirical distribution of the design) in the present regression context. However, extending contraction rates to the ``integrated" $L^2$ norm usually requires more work, and has previously been done for instance under sufficient smoothness or boundedness assumptions, which may not necessarily hold. In this work we show that, for classes of priors based on random basis expansions or Gaussian processes with RKHS of Sobolev type and in the random design setting, such $L^2$ posterior contraction rates can be obtained under substantially weaker assumptions than those currently used in the literature. Importantly we do not require a known a priori upper bound on its supremum norm or that its smoothness is larger than $d/2$, where $d$ is the dimension of the covariates. Our proof crucially relies on an application of the matrix Bernstein concentration inequality to empirical inner product matrices, which require explicit upper bounds on the basis functions at hand that we prove in several cases of interest. In particular we obtain upper bounds on the supremum norm of Mercer eigenfunctions of several reproducing kernels (including several Matérn kernels) which are of independent interest.
	\end{abstract}
	
	\tableofcontents
	
	\section{Introduction}\label{section:intro}
	
	In this paper we consider the following nonparametric regression model with random design :
	
	\begin{model}\label{model:nonparametric_regression}
		\[
		y_i = f_0(x_i) + \varepsilon_i, \quad i = 1,\ldots,n,
		\]
		where $n \geq 1, \quad \p{x_i,y_i} \iid P_0$ on $\XX \times \RR$ for some input domain $\XX$. We write $f_0(x) := E_0 \c{y|x}$ for the regression function of $y$ over $x$ and $\mu_0(dx) := \int_{y \in \RR} P_0\p{dxdy}$ the marginal distribution of $x$.
	\end{model}
	
	Model \ref{model:nonparametric_regression} is fundamental in modern statistics. Numerous estimation procedures have been proposed in order to recover the unknown regression function $f_0 \in \FF \subset \RR^\XX$ based only on the observations $\Xbb^n = \p{X_i,Y_i}_{i=1}^n$. Among them, Bayesian nonparametric methods have received much attention in recent years since the development of the general testing approach for quantification of posterior contraction rates in generic, potentially non-conjugate settings as initiated in the seminal works of \cite{ghosalConvergenceRatesPosterior2000} \& \cite{ghosalConvergenceRatesPosterior2007}, with applications to model \ref{model:nonparametric_regression} found for instance in \cite{castilloThomasBayesWalk2014,vaartRatesContractionPosterior2008,vaartInformationRatesNonparametric2011} or \cite{shenAdaptiveBayesianProcedures2015}. More precisely, under the frequentist model \ref{model:nonparametric_regression}, given a prior $\Pi$ on $\FF$ and a semi-metric $\rho$ on $\FF$, we say that a sequence $\varepsilon_n > 0$ is a \textit{posterior contraction rate} at $f_0$ with respect to $\rho$ if, for some $M>0$,
	\begin{equation}\label{def:posterior_contraction_rate}
		\EE_0 \Pi \c{ f : \rho\p{f,f_0} > M\varepsilon_n | \Xbb^n } \xrightarrow[n \to \infty]{} 0,
	\end{equation}
	where $\EE_0$ denotes the expectation under the probability distribution $\Pro_0$ of $\Xbb^n$ and $\Pi \c{\cdot | \Xbb^n}$ the posterior distribution which in this setting is given by Baye's formula
	\begin{equation}\label{eqn: Bayes_formula}
		\Pi \c{df | \Xbb^n} = \frac{\prod_{i=1}^n L\p{Y_i | f(X_i)} \Pi\c{df}}{\int_\FF \prod_{i=1}^n L\p{Y_i | f(X_i)} \Pi\c{df}},
	\end{equation}
	for some likelihood $L$. The testing approach then provides sufficient conditions on the relationship between $\Pi,\rho$ and $f_0$ in order to prove the existence of $\varepsilon_n$ satisfying \ref{def:posterior_contraction_rate}. It relies on a prior mass condition (a lower bound on the prior probability of neighbourhoods of $f_0$ formulated in the Kullback-Leibler sense), the existence of a subset $\FF_n \subset \FF$ containing most of the prior probability and the existence of a test $\phi_n : \Xbb^n \to \phi_n\p{\Xbb^n} \in \c{0,1}$ between the null hypothesis $f = f_0$ and the alternative $f \in \FF_n \cap \b{\rho(f,,f_0) > M\varepsilon_n}$ having appropriately decaying type I and type II errors, see\cite{ghosalFundamentalsNonparametricBayesian2017} or \cite{castilloBayesianNonparametricStatistics2024} for a recent introduction to the theory. The existence of the test $\phi_n$ is often deduced from the control of the $\rho-$metric entropy of $\FF_n$ and a structural requirement on $\rho$ commonly referred to as the \textit{testing condition} : for some fixed $A,K,\xi > 0$ for any $f_1 \in \FF$ with $\rho(f_0,f_1) > \varepsilon$, there exists a test $\phi$ satisfying
	\begin{equation}\label{eqn:testing_condition}
		\EE_{f_0} \c{\phi} \leq Ae^{-Kn\varepsilon^2} \text{ and } \sup_{\rho(f,f_1) < \xi \varepsilon} \EE_f \c{1-\phi} \leq Ae^{-Kn\varepsilon^2}.
	\end{equation}
	As a result, the testing condition limits the possible choice of metrics $\rho$ with respect to which posterior contraction rates \ref{def:posterior_contraction_rate} can be proven with the testing approach. In the context of nonparametric regression and for a Gaussian likelihood $L\p{y|f(x)} \propto \exp \p{-\p{y - f(x)}^2/2\sigma^2}, \quad \sigma > 0$, possible choices of semi-metrics satisfying the testing condition include the empirical $L^2$ norm
	\[
	\rho(f,f_0) = \norm{f-f_0}_n := \p{\frac{1}{n} \sum_{i=1}^n \p{f(X_i) - f_0(X_i)}^2}^{1/2},
	\]
	and, assuming an \iidt sampled design $X_i \iid \mu_0$, the Hellinger distance
	\[
	\rho(f,f_0) = H\p{f,f_0} = \int \p{1-\exp\p{-\frac{\p{f(x) - f_0(x)}^2}{8\sigma^2}}} \mu_0(dx),
	\]
	see \cite{ghosalFundamentalsNonparametricBayesian2017} Lemma 8.27 and Appendix D for the construction of appropriate tests (see also Section $8.5.2$ in \cite{ghosalFundamentalsNonparametricBayesian2017} for possible extensions to non Gaussian likelihoods). In the \iidt sampled design $X_i \iid \mu_0$, another natural choice for $\rho$ is the integrated version of $\norm{\cdot}_n$, namely the $L^2\p{\mu_0}$ norm
	\[
	\rho(f,f_0) = \norm{f-f_0}_{L^2\p{\mu_0}} := \p{\int \p{f(x) - f_0(x)}^2 \mu_0(dx)}^{1/2},
	\]
	for which the testing condition is not easily seen to hold (unless for instance one with priors supported on sets of uniformly bounded regression functions, see \cite{huangConvergenceRatesPosterior2004}). Obviously, if for some fixed $M>0$ we have $\norm{f_0}_\infty$ and $\Pi \c{\norm{f}_\infty \leq M | \Xbb^n} = 1$ for $\Pro_0-$amost every $\Xbb^n$, then $\Pi \c{\norm{f-f_0}_{L^2\p{\mu_0}} \leq C H\p{f,f_0} | \Xbb^n} = 1$ (where $C = \sup_{\abs{y} \leq 2M} \abs{y}/\sqrt{1-e^{-y^2/8\sigma^2}}$), and therefore a Hellinger posterior contraction rate is also an $L^2\p{\mu_0}$ one up to a constant multiple. The condition $\Pi \c{\norm{f}_\infty \leq M | \Xbb^n} = 1$ is certainly satisfied if $\Pi \c{\norm{f}_\infty \leq M} = 1$, and as a consequence the problem of optimal $L^2\p{\mu_0}$ contraction rates does not arise for priors supported on sets of uniformly bounded regression functions ; importantly this is not the case for Gaussian processes or random series with unbounded coefficients.
	
	In this spirit, if a known a priori upper bound on $\norm{f_0}_\infty$  is available, \cite{huangConvergenceRatesPosterior2004} shows that suitable priors (that are in particular supported on a set of uniformly bounded regression functions) do lead to optimal $L^2\p{\mu_0}$ contraction rates (although the proof is quite different than the simple argument presented above), and in an analogous way \cite{yangBayesianManifoldRegression2016,tangAdaptiveBayesianRegression2025} show that correctly clipped posterior distributions also possess similar optimal contraction properties. Beyond boundedness, most approaches require a minimal smoothness on $f_0$ in order to guarantee an optimal $L^2\p{\mu_0}$ posterior contraction rates. More precisely, Bernstein's concentration inequality has been used in \cite{vaartInformationRatesNonparametric2011,rosaPosteriorContractionRates2023,rosaNonparametricRegressionRandom2024a} in order relate the $\norm{\cdot}_n$ and the $\norm{\cdot}_{L^2\p{\mu_0}}$ and get optimal $L^2\p{\mu_0}$ contraction rates, while however requiring $s > d/2$ in order for the proof to go through, where here $s$ is the regularity of $f_0$ and $d$ the dimension of the input space $\XX$. Optimal $L^\infty\p{\mu_0}$ contraction rates (and therefore optimal $L^2\p{\mu_0}$ rates up to logarithmic factors) have also been derived in \cite{yooSupremumNormPosterior2016} for Gaussian B-splines priors by using the explicit representation of the posterior distribution (a consequence of conjugacy between Gaussian prior and likelihood), but still requires $s > d/2$. In \cite{xieRatesContractionRespect2019}, optimal $L^2\p{\mu_0}$ contraction rates for random series and Gaussian splines priors have been obtained by showing that the testing condition does hold for the $L^2\p{\mu_0}$ metric under specific assumptions on the prior support, which unfortunately rule out the $s \leq d/2$ case. Finally, it should be noted that optimal supremum norm contraction rates have been obtained in \cite{castilloBayesianSupremumNorm2014a} for (potentially non conjugate) product priors in the Gaussian white noise model for arbitrary positive regularities, but the results are not applicable to the nonparametric regression model \ref{model:nonparametric_regression}. Finally the problem of supremum norm contraction rates in the nonparametric regression model \ref{model:nonparametric_regression} has been investigated in \cite{yangFrequentistCoverageSupnorm2017}, where explicit rates are derived for (appropriately rescaled) Gaussian process priors with kernels having polynomially decaying Mercer eigenvalues, and thus the reproducing kernel Hilbert spaces (RKHSs) can be thought of as of Sobolev type. The smoothness conditions on $f_0$ that are used there are however not easy to connect with more classical function space norms. Indeed, even if by interpolation the $\alpha-$Sobolev norm as defined in Equation $(16)$ of the corresponding paper could potentially be seen to be equivalent to classical $\alpha-$Sobolev norms in concrete setting, an application of Theorem $3.1$ of the same paper yields the optimal $L^2\p{\mu_0}$ contraction rates for $s-$Sobolev functions ($n^{-\frac{s}{2s+d}}$ up to logarithmic factors) only if one assumes $s+d/2-$Sobolev smoothness of $f_0$. The problem does not appear however when using the results of the same paper regarding H\"older smooth functions, but the definition of the H\"older norm given in Equation $(17)$ is not classical, and does not appear to behave well under interpolation as in the Sobolev case. Furthermore, the conditions of applicability of Theorem $3.1$ in \cite{yangFrequentistCoverageSupnorm2017} necessitate in particular uniform boundedness of the Mercer eigenfunctions of the kernel operator (in the sense that $\sup_{j \geq 1} \norm{e_j}_\infty < +\infty$ where $\p{e_j}_{j \geq 1}$ are the Mercer eigenfunctions), and as we will see in Section \ref{sec:gaussian_processes}, this situation is not generic. The proof strategy used in \cite{yangFrequentistCoverageSupnorm2017} is however simple but powerful, and we will use it again in Section \ref{sec:gaussian_processes} : using conjugacy of the model we can explicitly write the posterior Gaussian process as another Gaussian process centered around a kernel ridge regression estimator with updated variance, and then control each term in the bias variance decomposition separately.\\
	
	In this note, we show how to obtain optimal $L^2\p{\mu_0}$ posterior contraction rates for two classes of priors : Gaussian processes with Sobolev type RKHS and random series. For Gaussian processes we follow the proof strategy of \cite{yangFrequentistCoverageSupnorm2017} (explicit representation for the posterior by conjugacy and connection with kernel ridge regression) and the results on kernel ridge regression estimators obtained in \cite{fischerSobolevNormLearning2020a}. The proof crucially uses an application of the matrix Bernstein's concentration inequality to the empirical inner product matrix of sufficiently many Mercer eigenfunctions, and doing so requires new (to the best of the author's knowledge) bounds on the Mercer eigenfunctions of the corresponding kernel operators which can be of independent interest. For random series priors $\Pi$ of the form $f \sim \sum_{j \geq 1} c_j e_j, \quad c \sim \tilde \Pi$, we can use the same matrix concentration inequality to directly relate the $\norm{f-f_0}_n$ and $\norm{f-f_0}_{L^2\p{\mu_0}}$ by explicit factors (even under non Gaussian priors) under a mild requirement on the chosen basis functions $\p{e_j}_{j \geq 1}$. In order to do so we require in particular for the posterior distribution to satisfy $\Pi \c{\norm{\sum_{j > J_n} c_j e_j}_\infty > M \varepsilon_n| \Xbb^n} = o_{\Pro_0}(1)$ where $\varepsilon_n$ is the desired contraction rate, for some $J_n \asymp \frac{n}{\ln n}$ (see Section \ref{section:random_series_priors} for precise statements). As we will see, although not universal these assumptions are not hard to meet in practice.\\
	
	The rest of the paper is organised as follows : we start by investigating the case of Gaussian process priors in Section \ref{sec:gaussian_processes} by using the functional analytical framework from \cite{fischerSobolevNormLearning2020a}. After stating a generic abstract theorem in Section \ref{subsection:general_gps} we apply it to the concrete case of Gaussian processes with RKHS (norm equivalent to) Sobolev spaces on domains and on compact Dirichlet spaces in Section \ref{subsection:examples_GPs}. Then we focus on random series priors in Section \ref{section:random_series_priors} where similarly after stating a generic theorem in Section \ref{subsec:general_results_random_series} we apply it to different cases of interest (sieve priors in Section \ref{subsection:sieve_priors} and hierarchical Gaussian processes on compact Dirichlet spaces in Section \ref{subsection:hierarchical_GPs_dirichlet_spaces}). In Section \ref{subsection:extension_rggs} we also adapt the results of Section \ref{subsection:sieve_priors} to the case of graph based semi supervised learning, where the basis functions depend on the covariates. Finally, Section \ref{section:discussion} ends with a discussion.\\
	
	\textbf{Notations :} $\RR$ will denote the set of real numbers, $\Nbb$ the set of natural number $\b{0,1,\ldots}$ and $\Nbb^* := \Nbb \backslash \b{0}$. For two Banach spaces $\p{E,\norm{\cdot}_E}$ and $\p{F,\norm{\cdot}_F}$ we will write $E \cong F$ if they are norm equivalent, i.e. if there exists a linear isomorphism $\varphi : E \to F$ and two constants $c_1,c_2>0$ such that for all $x \in E$ we have $c_1 \norm{\varphi(x)}_F \leq \norm{x}_E \leq c_2 \norm{\varphi(x)}_F$. For $p \in [1,\infty) \cup \b{\infty}$ and $\theta \in \p{0,1}$, we denote by $\p{E,F}_{\theta,p}$ the real interpolation space between $E$ and $F$, see e.g. \cite{lunardiInterpolationTheory2018}. If $E \subset F$ are such that the canonical injection $\iota : E \to F$ is continuous, we write $E \hookrightarrow F$ and $\norm{E \hookrightarrow F}$ for the operator norm of $\iota$, i.e. $\norm{E \hookrightarrow F} = \sup \b{\norm{x}_F : \norm{x}_E \leq 1}$. For a measured space $\p{\XX,\mu}$ and $p \geq 1$, the space $L^p\p{\mu}$ is then defined as the Banach space of equivalence classes of functions $f$ satisfying $\int \abs{f}^p d\mu  +\infty$ under the equivalence relation of $\mu-$almost everywhere equality. For any $d \geq 1$ and vector $v \in \RR^d$, we will write $\norm{v}_{\RR^d}$ for its Euclidean norm i.e. $\norm{v}_{\RR^d}^2 = \sum_{l=1}^d v_l^2$. For $n,p \geq 1$ we define $\RR^{n \times p}$ the space of real valued matrices of size $n \times p$, and for any $A \in \RR^{n \times p}$ we denote by $\norm{A}$ its operator norm, i.e. $\norm{A} = \sup \b{\norm{Ax}_{\RR^n} : \norm{x}_{\RR^p} \leq 1}$. For $A \in \RR^{n \times p}$ we define $\lambda_{\min}\p{A}$ (respectively $\lambda_{\max}\p{A}$) as the minimal (respectively maximal) eigenvalue of $A$, if it exists. In particular if $A$ is symmetric positive semidefinite we have $\lambda_{\max}\p{A} = \norm{A}$. For a linear operator $A : E \to F$ between two vector spaces we will use the notation $\rk \p{A}$ for its rank, i.e. $\rk \p{A} = \dim A(E)$. For a function $k : \RR^d \to \RR$ we will write $\hat k$ for its Fourier tranform, if it exists. If $X$ is a real valued random variable, we define $\norm{X}_{\psi_1} = \sup \b{\frac{1}{p}\EE \c{\abs{X}^p}^{1/p} : p \in \mathbb N^*}$ its $\psi_1-$norm, see e.g. \cite{vershyninHighDimensionalProbabilityIntroduction2018}. If $\p{X,Y}$ is a random vector with $Y$ being real valued, for any $x$ we define $\norm{Y | X = x}_{\psi_1}$ as $\norm{Y | X = x}_{\psi_1} = \sup \b{\frac{1}{p} \EE \c{\abs{Y}^p | X = x} : p \in \Nbb^*}$. For two probability measure $P,Q$ on a single measurable space $\XX$, we define $KL\p{P,Q} = +\infty$ if $P$ is anot absolutely continuous with respect to $Q$ and $KL\p{P,Q} = \int dP \ln \frac{dP}{dQ}$ otherwise. Similarly if $P$ is absolutely continuous with respect to $Q$ and $V\p{P,Q} = \int dP \ln^2 \frac{dP}{dQ} < +\infty$ we define $V_{2,0}\p{P,Q} = \int \p{\ln \frac{dP}{dQ} - KL\p{P,Q}}^2 dP$. Finally if $\p{\XX,\mu}$ is a measured space and $f$ a measurable function over $\XX$, we define $\essinf_\mu f$ as  the essential infimum of $f$ over $\XX$, i.e. $\essinf_\mu f = \sup \b{C \in \RR : C \leq f~\mu-~\text{a.e.}}$, and similarly $\esssup_\mu f = \inf \b{C \in \RR : C \geq f~\mu-~\text{a.e.}}$ for the essential supremum. Finally, if $f$ and $g$ are two real valued functions (including sequences), we write $f \asymp g$ (respectively $f = O(g), f \lesssim g,f \gtrsim g$) if there exists $c_1,c_2 > 0$ such that $c_1 \abs{f} \leq \abs{g} \leq c_2 \abs{f}$ (respectively $f \leq c_1 g,~f \geq c_1 g$), and $f = o(g)$ if $f/g \to 0$, all in the relevant asymptotic.
	
	\section{Gaussian processes}\label{sec:gaussian_processes}
	
	In this section we present results on $L^2\p{\mu_0}$ posterior contraction for conjugate Gaussian process regression. After introducing the functional analytical framework, we state a general theorem (Theorem \ref{thm:general_thm_GPs}) that we then apply to two different cases of interest (Corollaries \ref{corollary:matern_domains} \& \ref{corollary:matern_compact_dirichlet_spaces}).
	
	\subsection{General theorem}\label{subsection:general_gps}
	
	We adopt the framework of \cite{fischerSobolevNormLearning2020a}. We fix a separable RKHS $\Hbb$ on $\XX$ with a bounded and measurable kernel $k$ and we consider the canonical injection $I_{\mu_0} : \Hbb \to L^2\p{\nu}$ as well the integral operator
	
	\[
	S_{\mu_0} := I_{\mu_0}^* : \p{\begin{array}{ccc}
			L^2\p{\mu_0} & \to & \Hbb \\
			f & \mapsto & \int k(\cdot,y)f(y)\mu_0(dy)
	\end{array}}.
	\]
	
	We also define $T_{\mu_0} := I_{\mu_0} S_{\mu_0} : L^2\p{\mu_0} \to L^2\p{\mu_0}$ and $C_{\mu_0} := S_{\mu_0} I_{\mu_0} : \Hbb \to \Hbb$. Following \cite{fischerSobolevNormLearning2020a} and the references therein, the injection $I_{\mu_0}$ is Hilbert-Schmidt and thus compact, and so are the compositions $T_{\mu_0} = I_{\mu_0} S_{\mu_0}$ and $C_{\mu_0} = S_{\mu_0} I_{\mu_0}$. The spectral theorem for compact self-adjoint operator then implies
	\[
	T_{\mu_0} = \sum_{j \geq 1} s_j \inner{\cdot|e_j}_{L^2\p{\mu_0}} e_j, \quad \text{and} \quad C_{\mu_0} = \sum_{j \geq 1} s_j^2 \inner{\cdot|e_j}_{\Hbb} e_j,
	\]
	for some $s_j \geq 0$ and an orthonormal basis $\p{e_j}_{j \geq 1}$ of $\overline{I_{\mu_0}\p{\Hbb}}$. Here $\inner{\cdot|\cdot}_{L^2\p{\mu_0}}$ and $\inner{\cdot|\cdot}_\Hbb$ denote the inner products in the Hilbert spaces $L^2\p{\mu_0}$ and $\Hbb$, respectively. We then define the power spaces $\Hbb_{\mu_0}^\tau$ as
	\[
	\Hbb_{\mu_0}^\tau := \b{\sum_{j \geq 1} a_j s_j^{\tau/2} e_j : a \in \ell^2\p{\mathbb{N}^*}} \subset L^2\p{\nu},
	\]
	equipped with the norm
	\[
	\norm{\sum_{j \geq 1} a_j s_j^{\tau/2} e_j}_{\Hbb_{\mu_0}^\tau} := \norm{a}_{\ell^2\p{\mathbb{N}^*}}.
	\]
	In the case $\tau = 1$ we have $\Hbb_{\mu_0}^1 = \Hbb$ and
	\[
	\norm{\sum_{j \geq 1} a_j s_j^{1/2} e_j}_\Hbb = \norm{a}_{\ell^2\p{\mathbb{N}^*}}.
	\]
	Moreover when $0 < \tau < 1$ the space $\Hbb_{\mu_0}^\tau$ can be described by real interpolation
	\begin{equation}\label{eqn:real_interpolation}
		\Hbb_\nu^\tau \cong \p{L^2\p{\mu_0},\Hbb}_{\tau,2},
	\end{equation}
	see e.g. Theorem 4.6 in \cite{steinwartMercersTheoremGeneral2012}. As in \cite{fischerSobolevNormLearning2020a} we make the following assumptions :
	
	\begin{assumption}\label{assumption:EVD}
		There exists $0 < p < 1$ and constants $c_1,c_2$ such that
		\[
		c_1 j^{-1/p} \leq s_j \leq c_2 j^{-1/p}, \quad j \geq 1.
		\]
	\end{assumption}
	
	\begin{assumption}\label{assumption:EMB}
		For some $0 < \tau \leq 1$ we have
		\[
		\norm{\Hbb_{\mu_0}^\tau \hookrightarrow L^\infty\p{\mu_0}} < +\infty.
		\]
	\end{assumption}
	
	\begin{assumption}\label{assumption:SRC}
		$f_0 \in L^\infty\p{\mu_0}$ and for some $0 < \beta \leq 2$ we have
		\[
		\norm{f_0}_{\Hbb_{\mu_0}^\beta} < +\infty.
		\]
	\end{assumption}
	
	\begin{assumption}\label{assumption:MOM}
		For some $L,\sigma > 0$ we have
		\[
		\int \abs{y - f_0(x)}^m P_0(dxdy) \leq \frac{\sigma^2}{2} m! L^{m-2}
		\]
		for all $m \geq 2$ and $\mu_0-$almost every $x \in \XX$. We also assume that $\abs{P_0}_2^2 := \int y^2 P_0\p{dxdy} < +\infty$.
	\end{assumption}
	
	Let us comment on these assumptions. As we will see in concrete settings \ref{corollary:matern_domains} \& \ref{corollary:matern_compact_dirichlet_spaces}, by real interpolation \ref{eqn:real_interpolation} the space $\Hbb_\nu^\tau$ for $\tau \in \p{0,1}$ will be seen to be norm equivalent to the Sobolev space $H^{\tau\p{\alpha + d/2}}\p{\XX}$ for some $\alpha > 0$ such that $\Hbb \cong H^{\alpha + d/2}\p{\XX}$, where $H^s\p{\XX}$ denotes the $s-$order Sobolev space over $\XX$ for all $s > 0$. The proof of Corollaries \ref{corollary:matern_domains} \& \ref{corollary:matern_compact_dirichlet_spaces} will allow us to take $p = \frac{d}{2\alpha + d}$ where $d$ is the dimension of $\XX$ and $\tau > p$ as close to $p$ as desired. As a consequence, Assumption \ref{assumption:EMB} will be a consequence of the Sobolev embedding theorem (which is classical in the setting of Corollary \ref{corollary:matern_domains}, but for which we provide a proof for compact Dirichlet spaces, see Proposition \ref{prop:embeddings}). The proof of Corollaries \ref{corollary:matern_domains} \& \ref{corollary:matern_compact_dirichlet_spaces} will also allow us to take $\beta = \frac{2s}{2\alpha + d}$ for some $s>0$, so that Assumption \ref{assumption:SRC} is actually simply the requirement $f_0 \in L^\infty \p{\nu} \cap H^s\p{\XX}$, i.e. a boundedness and $s-$Sobolev smoothness assumption on $f_0$. Importantly we do not require $s > d/2$. Assumption \ref{assumption:MOM} allows for heteroscedasticity of the errors $y - f_0(x)$ ; a reformulation is
	\[
	\norm{\frac{y-f_0(x)}{L}}_{L^m\p{P_0}} \leq \p{\frac{\sigma^2}{2L^2} m!}^{1/m} \asymp m
	\]
	for any $m \geq 2$, where we have used Stirling's approximation. Thus $P_0$ satisfies the first bound in Assumption \ref{assumption:MOM} for some $L,\sigma > 0$ if and only if the errors $y-f_0(x)$ are subexponential, see e.g. \cite{vershyninHighDimensionalProbabilityIntroduction2018}. If
	\begin{equation}\label{eqn:condition_subexp}
	\int \norm{y-f_0(x) | X = x}_{\psi_1}^m \mu_0(dx) \leq c^m
	\end{equation}
	for some $c>0$, then
	\[
	\int \p{y-f_0(x)}^m P_0(dxdy) =  \int \int \p{y-f_0(x)}^m P_0(dy|x) \mu_0(dx) \leq \int \p{\norm{y-f_0(x) | x}_{\psi_1} m}^m \mu_0(dx) \leq \p{cm}^m,
	\]
	so that the first bound of Assumption \ref{assumption:MOM} does hold in this case. The condition \ref{eqn:condition_subexp} is in particular guaranteed whenever $x \mapsto \norm{y-f_0(x)}_{\psi_1} \in L^\infty\p{\mu_0}$, for instance Gaussian or Laplace errors with essentially bounded standard deviations. Moreover using $y^2 = \p{y-f_0(x) + f_0(x)}^2 \leq 2\p{y-f_0(x)}^2 + 2f_0(x^2) \leq 2\p{y-f_0(x)}^2 + 2\norm{f_0}_\infty^2$ we get
	\[
	\abs{P_0}^2 \leq 2\norm{f_0}_\infty^2 + 2 \norm{y-f_0(x)}_{L^2\p{P_0}}^2 \lesssim \norm{f_0}_\infty^2 + \norm{y-f_0(x)}_{\psi_1}^2,
	\]
	and therefore Assumption \ref{assumption:SRC} and subexponentiality of $y-f_0(x)$ directly implies $\abs{P_0}_2^2 <+\infty$. Finally, Assumption \ref{assumption:EVD} is a structural property that must be satisfied by the kernel $k$ and rules out exponentially decaying eigenvalues (see e.g. heat kernels on compact manifolds or Dirichlet spaces \cite{devitoReproducingKernelHilbert2021a,castilloThomasBayesWalk2014}), but as we will see typically holds if the RKHS $\Hbb$ is of Sobolev type. More precisely, in our examples \ref{corollary:matern_domains} \& \ref{corollary:matern_domains} we will (following \cite{fischerSobolevNormLearning2020a}) verify Assumption \ref{assumption:EVD} by only using that the RKHS $\Hbb$ is norm equivalent to a classical Sobolev space and leveraging the connection with the approximation numbers of the embedding $\Hbb \hookrightarrow L^2\p{\mu_0}$ (see e.g. Chapter $3$ in \cite{edmundsFunctionSpacesEntropy1996} and Chapter $4$ in \cite{carlEntropyCompactnessApproximation1990}).\\
	
	Under Assumptions \ref{assumption:EVD} \ref{assumption:EMB} \ref{assumption:SRC} \& \ref{assumption:MOM}, \cite{fischerSobolevNormLearning2020a} proved sharp convergence rates for the kernel regression estimator associated with $\Hbb$ and $k$ under several Sobolev-type metrics. We generalise here the result to Gaussian random elements with RKHS $\Hbb$ and kernel $k$ by a careful analysis of the posterior variance, which is explicit thanks to conjugacy.
	
	\begin{theorem}\label{thm:general_thm_GPs}
		Let $f \sim \Pi$ be a mean zero Gaussian random element in $L^2\p{\nu}$ with RKHS $\Hbb$. Consider the posterior distribution
		\[
		\Pi \c{df | X,Y} \propto \prod_{i=1}^n \p{2\pi \sigma^2}^{-1/2} \exp \p{- \frac{\p{y_i - f(x_i)}^2}{2\sigma^2}} \Pi \c{df},
		\]
		corresponding to the observation model
		\[
		y_i = f(x_i) + \varepsilon_i, \quad \varepsilon_i 
		\iid \NN\p{0,\sigma^2},
		\]
		for some $\sigma > 0$. Then if Assumptions \ref{assumption:EVD} \ref{assumption:EMB}  \ref{assumption:SRC} and \ref{assumption:MOM} hold with either $0 < p < \tau < \beta < \beta + p \leq 1$ or $0 < \beta \leq p < \tau < \beta + p \leq 1$, for any sequence $M_n \to \infty$ we have
		\[
		\EE_0 \c{\norm{f - f_0}_{L^2\p{\nu}} > M_n \varepsilon_n | \Xbb^n} \to 0, \quad \varepsilon_n = n^{-\frac{\beta}{2}}.
		\]
	\end{theorem}
	
	As we will se in the proof of the examples \ref{corollary:matern_domains} \& \ref{corollary:matern_compact_dirichlet_spaces}, the regime $0 < p < \tau < \beta < \beta + p \leq 1$ corresponds to the case $s > d/2$ while $0 < \beta \leq p < \tau < \beta + p \leq 1$ corresponds to $s \leq d/2$. The proof of Theorem \ref{thm:general_thm_GPs} (which can be found in details in Section \ref{appendix:proof_GPs}) starts from the bias variance decomposition, where the bias is identified with the $L^2\p{\mu_0}$ risk of the kernel ridge regression estimator $\hat f$ associated with $k$, which is bounded using the results in \cite{fischerSobolevNormLearning2020a} (although a small refinement of the proof is required in order to handle the case $s \leq d/2$). For the posterior variance, using technical but straightforward linear algebra we are able to show that, for any $L \geq 1$,
	\[
	\ind{E_L} \int \norm{\hat f - f}_{L^2\p{\mu_0}}^2 \Pi \c{df | \Xbb^n} \leq \frac{\sigma^2 L}{n \kappa} + \sum_{l > L} \mu_l,
	\]
	where $E_L$ is the event $\b{\frac{1}{n} U_L^T U_L \geq \kappa I_L}$, for $U_L$ the $n \times L$ matrix with entries $\p{U_L}_{il} = e_l(x_i)$. The last term in the right hand side of the last display is controlled using Assumption \ref{assumption:EVD}, and in order to show that $\Pro_0\p{E_L} \to 1$ we apply the matrix Bernstein concentration inequality, see Theorem \ref{theorem:1} in Section \ref{subsec:general_results_random_series} (with $L \asymp n^p$). In order to do so we crucially have to find an almost sure upper bound on the quantity $\sup_{x \in \XX} \sum_{l = 1}^L e_l^2(x)$, which we find by exploiting Assumptions \ref{assumption:EVD} \& \ref{assumption:EMB}, see Lemma \ref{lemma:generic_matern_bound_basis}.\\
	
	Before applying Theorem \ref{thm:general_thm_GPs} to two concrete examples, let us mention a byproduct of our analysis which is of independent interest. As detailed above, the proof of Theorem \ref{thm:general_thm_GPs} required to find an almost sure bound on $\sup_{x \in \XX} \sum_{l = 1}^L e_l^2(x)$ for some $L \asymp n^p$. Lemma \ref{lemma:generic_matern_bound_basis} provides such a bound : for any $\tau \in \p{p,\beta}$ (if $\beta > p$) or $\tau \in \p{p,\beta+p}$ (if $\beta \leq p$),
	\[
	\sup_{x \in \XX} \sum_{l = 1}^L e_l^2(x) \lesssim L^{\tau/p}, \quad L \geq 1,
	\]
	up to a comparison constant depending only on $\tau$ and $\mu_0$. In particular we obtain :
	
	\begin{proposition}\label{proposition:bound_mercer_eigenfunctions_domains_dirichlet_spaces}
		Under the conditions of Theorem \ref{thm:general_thm_GPs}, for any $\delta > 0$ there exists a constant $c>0$ that depends only on $\delta$ and $\mu_0$ such that
		\[
		\sup_{x \in \XX} \sum_{l = 1}^L e_l^2(x) \leq c^2 L^{1+\delta}, \quad L \geq 1.
		\]
		In particular for any $\delta > 0$ there exists $c>0$ such that $\norm{e_l}_\infty \leq c l^{1/2 + \delta}$ for all $l \geq 1$.
	\end{proposition}
	
	The bound $\norm{e_l}_\infty \leq c j^{1/2 + \delta}$ of Proposition \ref{proposition:bound_mercer_eigenfunctions_domains_dirichlet_spaces} is essentially sharp using only Assumptions \ref{assumption:EVD} \& \ref{assumption:EMB}. Indeed, on $\XX = \c{-1,1}^d$, $k$ the kernel
	\[
	k(x,y) := \sum_{l \geq 1} 2^{-2l(\alpha + d/2)}\sum_{k = 1}^{2^{ld}} \psi_{lk}(x)\psi_{lk}(y)
	\]
	for some $S-$regular, compactly and boundary corrected Daubechies wavelet basis on $\XX$ (see e.g. \cite{gineMathematicalFoundationsInfiniteDimensional2015} Chapter $4$), the associated kernel operator with respect to $\mu_0 = \mu$ the Lebesgue measure is diagonal in the basis $\p{\psi_{lk}}$ with eigenvalues $s_j = 2^{-2l(\alpha+d/2)}$. By relabeling $\p{l,k}$ into a new index $j \geq 1$ according to the lexicographic order we easily see that $s_j \asymp j^{-(1+2\alpha/d)}$. Indeed the eigenvalue associated with $\psi_{lk}$ is $2^{-2l(\alpha + d/2)}$, and therefore $j = 2^{ld}$ yields $2^{-2l(\alpha + d/2)} = j^{-(1+2\alpha/d)}$, while if $2^{ld} < j \leq 2^{(l+1)d}$ we obtain $2^{-(2\alpha + d)} 2^{-2l(\alpha + d/2)} \leq j^{-(1+2\alpha/d)} < 2^{-2l(\alpha + d/2)}$, i.e. Assumption \ref{assumption:EVD} holds with $p = \frac{d}{2\alpha + d}$. Moreover the norm of the Hilbert space $\Hbb_{\mu_0}^\tau$ is by definition given by
	\[
	\norm{f}_{\Hbb_{\mu_0}^\tau}^2 = \sum_{l \geq 1} 2^{2\tau l\p{\alpha + d/2}} \sum_{k = 1}^{2^{ld}} \inner{\psi_{lk} | f}_{L^2\p{\mu_0}}^2,
	\]
	which (since $\mu_0 = \mu$) is an equivalent norm on the Sobolev space $H^{\tau\p{\alpha + d/2}}\p{\XX}$, and we have a continuous embedding $\Hbb_{\mu_0}^\tau \hookrightarrow L^\infty\p{\mu_0}$ for $\tau \p{\alpha + d/2} > d/2 \iff \tau > p$, see \cite{gineMathematicalFoundationsInfiniteDimensional2015} Chapter $4$. But $\norm{\psi_{lk}}_\infty \asymp 2^{ld/2}$ which, according to the lexicographic relabeling, is of order $j^{1/2}$ (by the same argument as above that was used to show that $s_j \asymp j^{-(1+2\alpha/d)}$). Whether or not the conclusions of Proposition \ref{proposition:bound_mercer_eigenfunctions_domains_dirichlet_spaces} hold in general with $\delta = 0$ is however not clear. Obviously the conclusion $\norm{e_l}_\infty = O\p{l^{1/2 + \delta}}$ for all $\delta > 0$ can be suboptimal for specific basis, see the discussion in Remark \ref{remark:suboptimality_laplacian_eigenbasis}, but that appears to be irrelevant for proof of posterior contraction.
	
	\subsection{Examples}\label{subsection:examples_GPs}
	
	We start by applying Theorem \ref{thm:general_thm_GPs} to Gaussian processes on domains of $\RR^d$ with $\CC^\infty$ boundary (in the sense of \cite{triebelTheoryFunctionSpaces1983}, Section $3.2.1$).
	
	\begin{corollary}\label{corollary:matern_domains}
		In the setting of Section \ref{subsection:general_gps}, let $\XX$ be a non empty, open, connected and bounded domain of $\RR^d$ with $\CC^\infty-$boundary, $\mu_0$ have a density $p_0$ with respect to the Lebesgue measure $\mu$ on $\XX$ that satisfies $0 < \inf_\XX p_0 \leq \sup_\XX p_0 < +\infty$ and $\Hbb \subset L^2\p{\nu}$ be an RKHS with reproducing kernel $k$ that is equal as set and norm equivalent to the Sobolev space $H^{\alpha + d/2}\p{\XX}$ for some $\alpha > 0$, as defined in \cite{adamsSobolevSpaces2003}, Chapter $7$. Then if $f_0 \in L^\infty \p{\mu} \cap H^s\p{\XX}, s>0$, Assumption \ref{assumption:MOM} holds and $f$ is a mean zero Gaussian process on $\XX$ with covariance kernel $k$, with $\Pi \c{df | \Xbb^n}$ the posterior distribution
		\[
		\Pi \c{df | \Xbb^n} \propto \prod_{i=1}^n \p{2\pi \sigma^2}^{-1/2} \exp \p{- \frac{\p{y_i - f(x_i)}^2}{2\sigma^2}} \Pi \c{df},
		\]
		corresponding to the observation model
		\[
		y_i = f(x_i) + \varepsilon_i, \quad \varepsilon \iid \NN\p{0,\sigma^2}
		\]
		for some $\sigma > 0$, for any $M_n \to \infty$ we have
		\[
		\EE_0 \Pi \c{\norm{f-f_0}_{L^2\p{\nu}} > M_n \varepsilon_n | X,Y} \to 0, \quad \varepsilon_n = n^{-\frac{\alpha \wedge s}{2\alpha + d}}.
		\]
	\end{corollary}
	
	Corollary \ref{corollary:matern_domains} is obtained by an application of Theorem \ref{thm:general_thm_GPs} by setting $\beta = \frac{2s}{2\alpha +d}$ and $p = \frac{d}{2\alpha + d}$. As in \cite{fischerSobolevNormLearning2020a}, Assumption \ref{assumption:EVD} follows because the eigenvalues $s_j$ can be seen as the square of the approximation numbers of the embedding $\Hbb \hookrightarrow L^2\p{\mu_0}$ (see Equation $4.4.12$ in \cite{carlEntropyCompactnessApproximation1990}), which themselves are upper and lower bounded in \cite{edmundsFunctionSpacesEntropy1996} Chapter $3$. Moreover as in \cite{fischerSobolevNormLearning2020a} we have $\Hbb_\nu^\tau \cong H^{\tau(\alpha+d/2)}\p{\XX}$ for $0 < \tau < 1$, proving that Assumption \ref{assumption:EMB} holds (by the Sobolev embedding theorem, see \cite{adamsSobolevSpaces2003} Chapter $7$) and that Assumption \ref{assumption:SRC} holds as well.
	
	There are many Gaussian processes with RKHS $\Hbb$ satisfying $\Hbb \cong H^{\alpha + d/2}\p{\XX}$ : for instance Mat\'ern processes \cite{kanagawaGaussianProcessesReproducing2025}, but more generally any restriction to $\XX$ of a Gaussian process on $\RR^d$ with RKHS $\tilde \Hbb$ and kernel $\tilde k$ of the form $\tilde k(x,y) = F\p{\norm{x-y}_{\RR^d}}$ for some continuous function $F : [0,\infty) \to \RR$ such that
	\[
	c_1 \p{1 + \norm{\omega}_{\RR^d}^2}^{-(\alpha + d/2)} \leq \hat{\tilde k}\p{\omega} \leq c_2 \p{1 + \norm{\omega}_{\RR^d}^2}^{-(\alpha + d/2)}, \quad \omega \in \RR^d,
	\]
	satisfies $\tilde \Hbb \cong H^{\alpha + d/2}\p{\XX}$ (here by abuse of notations we have identified $\tilde k(x,y)$ and $\tilde k(x-y)$). Indeed, Lemma $5.1$ in \cite{yangBayesianManifoldRegression2016} shows that
	\[
	\norm{h}_\Hbb = \inf_{\substack{\tilde h \in \tilde \Hbb \\ \tilde h_{|\XX} = f}} \norm{\tilde h}_{\tilde \Hbb}, \quad h \in \Hbb,
	\]
	and the results of Chapter $5$ in \cite{triebelTheoryFunctionSpaces1992} imply that
	\[
	h \mapsto \inf_{\substack{\tilde h \in H^{\alpha + d/2}\p{\RR^d} \\ \tilde h_{|\XX} = h}} \norm{\tilde h}_{H^{\alpha + d/2}\p{\RR^d}}
	\]
	is actually an equivalent norm on $H^{\alpha + d/2}\p{\XX}$.\\
	
	Beyond domains of $\RR^d$, extending the conclusion of Corollary \ref{corollary:matern_domains} is of particular interest, for instance to Gaussian processes defined on compact manifolds or even more generally compact Dirichlet spaces that we now introduce, following closely \cite{coulhonHeatKernelGenerated2012,castilloThomasBayesWalk2014}. We consider a metric measured space $\p{\XX,\mu,\rho}$ satifying the following assumptions :
	
	\begin{assumption}\label{assumptions:dirichlet_spaces}
		$\p{\XX,d}$ is compact and has the volume regularity property : for some $C,c,d>0$,
		\[
		cr^d \leq \mu \p{B(x,r)} \leq C r^d, \quad x \in \XX, \quad r>0.
		\]
		We also assume that $\XX$ is equipped with a self adjoint positive linear operator $L$ defined on a dense domain $D \subset L^2 := L^2\p{\XX,\mu}$, and $e^{-tL}$ the semigroup of linear operators with infinitesimal generator $L$ consists of integral operators with kernels $p_t(\cdot,\cdot)$ with respect to the measure $\mu$\footnote{In the sense that $e^{-tL}f = \int f(y) p_t(\cdot,y)\mu(dy), \quad t > 0.$} satisfying :
		\begin{enumerate}
			\item small time Gaussian upper bounds :
			\[
			p_t(x,y) \leq \frac{C \exp \p{-c\rho^2\p{x,y}/t}}{\sqrt{\mu \p{B(x,\sqrt{t})} \mu \p{B(y,\sqrt{t})}}}, \quad x,y \in \XX, \quad 0 < t \leq 1,
			\]
			\item H\"older continuity : for some $\eta > 0$, and for every $0 < t \leq 1, \quad x,y,y' \in \XX$ with $\rho(y,y') \leq \sqrt{t}$,
			\[
			\abs{p_t(x,y) - p_t(x,y')} \leq C \p{\frac{\rho(y,y')}{\sqrt{t}}}^\eta \frac{\exp \p{-c\rho^2(x,y)/t}}{\sqrt{\mu\p{B(x,\sqrt{t})} \mu\p{B(y,\sqrt{t})}}},
			\]
			\item Markov property :
			\[
			\int p_t(x,y) \mu(dy) = 1, \quad x \in \XX, \quad t > 0.
			\]
		\end{enumerate}
	\end{assumption}
	
	The set of assumptions \ref{assumptions:dirichlet_spaces} implies the ones made in \cite{coulhonHeatKernelGenerated2012} and \cite{castilloThomasBayesWalk2014}. Moreover, the compactness of $\p{\XX,d}$ together with the set of assumptions \ref{assumptions:dirichlet_spaces} and Proposition \cite{coulhonHeatKernelGenerated2012} automatically imply $\mu \p{\XX} < +\infty$. The results in \cite{castilloThomasBayesWalk2014} were proven under the normalisation $\mu\p{\XX} = 1$, but readily extend to the more general case $\mu\p{\XX} < +\infty$. The property $\mu\p{\XX} < \infty$ in turn implies the non collapsing condition $(1.3)$ in \cite{coulhonHeatKernelGenerated2012}, as shown in Proposition $2.1$ of the same paper. Thus when working with the set of assumptions \ref{assumptions:dirichlet_spaces}, we may use all of the results from \cite{coulhonHeatKernelGenerated2012} that are valid in the compact case, as well as the ones from \cite{castilloThomasBayesWalk2014}. Several examples of spaces satisfying Assumption \ref{assumptions:dirichlet_spaces} are presented in \cite{coulhonHeatKernelGenerated2012,castilloThomasBayesWalk2014} and include compact manifolds but also $\XX = \c{-1,1}$ equipped with the Jacobi operator. Framing domains of $\RR^d$ as the ones studied in Corollary \ref{corollary:matern_domains} (and others with rougher boundaries) as particular cases of spaces satisfying the set of assumptions \ref{assumptions:dirichlet_spaces} is possible by choosing appropriate boundary conditions for the chosen operator $L$, see the discussion in Section $1.3.2$ of \cite{coulhonHeatKernelGenerated2012}. In particular the arguments developped here should extend to Gaussian processes on $\XX = \c{-1,1}^d$.\\

	As a consequence of the set of assumptions \ref{assumptions:dirichlet_spaces}, Proposition $3.20$ in \cite{coulhonHeatKernelGenerated2012} shows that the spectrum of $L$ is discrete and of the form $0 \leq \lambda_1 \leq \lambda_2 \leq \ldots$,
	\[
	L^2 = \overset{\perp}{\bigoplus}_{j \geq 1} E_{\lambda_j}, \quad E_{\lambda_j} = \ker \p{L - \lambda_j I} \text{ and } \dim \p{E_{\lambda_j}} < \infty.
	\]
	The orthogonal decomposition of $L^2$ in the last display has to be understood in the Hilbert sum sense, see e.g. Chapter 5 in \cite{brezisFunctionalAnalysisSobolev2011a}. Notice that contrariwise to \cite{coulhonHeatKernelGenerated2012}, the sequence $\p{\lambda_j}_{j \geq 1}$ corresponds to the spectrum of $L$ with multiplicities. According to Equation $3.51$ in \cite{coulhonHeatKernelGenerated2012}, the heat kernel $p_t$ has the following representation
	\[
	p_t(x,y) = \sum_{j \geq 1} \lambda_j u_j(x) u_j(y), \quad x,y \in \XX,
	\]
	with uniform convergence over $\XX^2$ and where $\p{u_j}_{j \geq 1}$ is an orthonormal basis of $L^2\p{\mu}$. The basis $\p{u_j}_{j \geq 1}$ can be used to define the Sobolev spaces $H^s\p{\XX}, \quad s>0$ over $\XX$ :
	
	\begin{definition}\label{definition:sobolev_spaces_dirichlet_spaces}
	For $\XX$ a space satisfying the set of assumptions \ref{assumptions:dirichlet_spaces} we define the Sobolev spaces $H^s\p{\XX}, \quad s > 0$ as the Hilbert subspaces of $L^2\p{\mu}$ of elements $f$ satisfying
	\[
	\norm{f}_{H^s\p{\XX}}^2 := \sum_{j \geq 1} \p{1 + \lambda_j}^s \inner{u_j|f}_{L^2\p{\mu}}^2 < + \infty.
	\]
	\end{definition}
	
	In the case where $\XX$ is a compact manifold, the Sobolev spaces $H^s\p{\XX}$ admit equivalent chart definitions, see e.g. \cite{devitoReproducingKernelHilbert2021a} and the references therein. We give in Section \ref{subsection:dirichlet_spaces} additional useful results on compact Dirichlet spaces. Using Definition \ref{definition:sobolev_spaces_dirichlet_spaces} we can now formulate a corollary of Theorem \ref{thm:general_thm_GPs} in this setting :
	
	\begin{corollary}\label{corollary:matern_compact_dirichlet_spaces}
		In the setting of Section \ref{subsection:general_gps}, let $\XX$ be a space satisfying the set of assumptions \ref{assumptions:dirichlet_spaces}, $\mu_0$ have a density $p_0$ with respect to $\mu$ that satisfies $0 < \inf_\XX p_0 \leq \sup_\XX p_0 < +\infty$ and $\Hbb \subset L^2\p{\mu}$ be an RKHS with reproducing kernel $k$ that is equal as set and norm equivalent to the Sobolev space $H^{\alpha + d/2}\p{\XX}$ for some $\alpha > 0$. Then if $f_0 \in L^\infty \p{\mu} \cap H^s\p{\XX}, s>0$, Assumption \ref{assumption:MOM} holds and $f$ is a mean zero Gaussian process over $\XX$ with covariance kernel $k$, with $\Pi \c{df | X,Y}$ the posterior distribution
		\[
		\Pi \c{df | X,Y} \propto \prod_{i=1}^n \p{2\pi \sigma^2}^{-1/2} \exp \p{- \frac{\p{y_i - f(x_i)}^2}{2\sigma^2}} \Pi \c{df},
		\]
		corresponding to the observation model
		\[
		y_i = f(x_i) + \varepsilon_i, \quad \varepsilon \iid \NN\p{0,\sigma^2}.
		\]
		we have, for any $M_n \to \infty$,
		\[
		\EE_0 \Pi \c{\norm{f-f_0}_{L^2\p{\mu_0}} > M_n \varepsilon_n | X,Y} \to 0, \quad \varepsilon_n = n^{-\frac{s \wedge \alpha}{2\alpha + d}},
		\]
	\end{corollary}
	
	The proof of Corollary \ref{corollary:matern_compact_dirichlet_spaces} can be found in Section \ref{appendix:proof_GPs}. It is similar to the one of Corollary \ref{corollary:matern_domains}, although in that case we need to directly estimate the approximation numbers of the embedding $\Hbb \hookrightarrow L^2\p{\mu_0}$, which we do by connecting these with the eigenvalues $\p{\lambda_j}_{j \geq 1}$ of $L$ for which we provide a growth estimate in Lemma \ref{proposition:growth_eigenpairs_dirichlet_spaces} (it is essentially a direct application of the results in \cite{coulhonHeatKernelGenerated2012}).
	
	As for Corollary \ref{corollary:matern_domains}, many Gaussian processes have an RKHS $\Hbb$ that satisfies $\Hbb \cong H^{\alpha + d/2}\p{\XX}$ : by construction the centered Gaussian process with kernel
	\[
	k(x,y) := \sum_{j \geq 1} \p{1 + \lambda_j}^{-(\alpha + d/2)} u_j(x) u_j(y)
	\]
	is one of them (as in that case we have exact equality $\norm{\cdot}_\Hbb = \norm{\cdot}_{H^{\alpha + d/2}\p{\XX}}$), but also
	\[
	k(x,y) := \sum_{j \geq 1} c_1\p{c_2 + c_3\lambda_j}^{-(\alpha + d/2)} u_j(x) u_j(y)
	\]
	for any choices of $c_1,c_2,c_3 > 0$, see Remark $13$ in \cite{rosaPosteriorContractionRates2023}, which readily extends from the setting of compact manifolds to general compact Dirichlet spaces satisfying \ref{assumptions:dirichlet_spaces}. In a more complicated way, if $\XX$ is a compact submanifold of $\RR^D$ of dimension $d < D$ then $\Hbb \cong H^{\alpha + d/2}\p{\XX}$ whenever $f$ is the restriction to $\XX$ of a Gaussian process $\tilde f$ on $\RR^D$ with RKHS $\tilde \Hbb$ that is norm equivalent to the Sobolev space $H^{\alpha + D/2}\p{\RR^D}$, for instance a Mat\'ern process on $\RR^D$ with smoothness parameter $\alpha > 0$ (see \cite{rosaPosteriorContractionRates2023}), or more generally arguing as above a centered Gaussian process $\tilde f$ on $\RR^D$ with kernel $\tilde k$ of the form $\tilde k(x,y) = F\p{\norm{x-y}_{\RR^d}}$ for some continuous function $F : [0,\infty) \to \RR$ and such that
	\[
	c_1 \p{1 + \norm{\omega}_{\RR^d}^2}^{-(\alpha + d/2)} \leq \hat{\tilde k}\p{\omega} \leq c_2 \p{1 + \norm{\omega}_{\RR^d}^2}^{-(\alpha + d/2)}, \quad \omega \in \RR^d.
	\]
	
	\section{Random series priors}\label{section:random_series_priors}
	
	The $L^2\p{\mu_0}$ posterior contraction rates obtained in Section \ref{sec:gaussian_processes} crucially used an application of the matrix Bernstein concentration inequality \ref{theorem:1} in order to lower bound the minimal eigenvalue of the empirical inner product matrix of sufficiently many Mercer eigenfunctions of the kernel operator associated with the Gaussian process at hand. In this section, we show how to simply extend the argument to priors based on random series but with potentially non Gaussian coefficients.
	
	\subsection{General reduction to $\norm{\cdot}_n-$contraction rates}\label{subsec:general_results_random_series}
	
	We consider a measured space $\p{\XX,\mu}$ as well as sequence of functions $\p{u_j}_{j = 1}^J$ (for some $J \geq 1$) and we define the constant $C_J > 0$ by
	\[
	C_J^2 := \norm{\frac{1}{J} \sum_{j=1}^J u_j^2}_{L^\infty\p{\mu}}.
	\]
	We notice that $C_J$ can be bounded by $O(1)$ (or almost) in various cases of interest :
	\begin{enumerate}\label{examples:C_J_random_series}
		\item If the $u_j's$ are uniformly bounded by $M>0$ (e.g. for $\XX = \c{0,1}$, $\p{u_j}_{j \geq 1}$ the trigonometric basis) we have $C_J \leq M^2$.
		\item If $\XX = \c{0,1}^d, d \geq 1$ and $\p{u_j}_{j \geq 1}$ is a wavelet basis on $\XX$ (see Section 4.3.6 in \cite{gineMathematicalFoundationsInfiniteDimensional2015}), then $C_J = O\p{1}$ Indeed, if we label this basis as $\p{\psi_{lk}}_{l \geq -1, k \in I_l}$ where $\# I_l = O\p{2^{ld}}$ then for any index $L$ and point $x \in \XX$ we have (using Proposition 4.2.8 in \cite{gineMathematicalFoundationsInfiniteDimensional2015})
		\[
		\sum_{l=1}^L \sum_{k \in I_l} \psi_{lk}(x)^2 \leq \norm{\sum_{l=1}^L \sum_{k \in I_l} \psi_{lk}(x)\psi_{lk}}_\infty
		\leq \sum_{l=1}^L \norm{ \sum_{k \in I_l} \psi_{lk}(x) \psi_{lk}}_\infty
		\lesssim \sum_{l = 1}^L 2^{ld/2} \max_{k \in I_l} \abs{\psi_{lk}(x)}
		\lesssim \sum_{l=1}^L 2^{ld}
		\lesssim 2^{Ld},
		\]
		which proves the result for $J = \sum_{l=1}^L \#I_l \asymp 2^{Ld}$. For intermediates $J = \sum_{l=1}^{L+1} \# I_{l'}, I_l' \subsetneq I_l$ we use $Q_J(x,x) \leq Q_{J'}(x,x) \lesssim 2^{(L+1)d} \lesssim 2^{Ld} \leq J$, for $J' = \sum_{l=1}^{L+1} \# I_l$.
		\item If $\XX$ is a compact Dirichlet space that satisfies the set of assumptions \ref{assumptions:dirichlet_spaces} as in Section \ref{subsection:examples_GPs} and $\p{u_j}_{j \geq 1}$ the associated eigenbasis of the operator $L$, then we show in Proposition \ref{proposition:growth_eigenpairs_dirichlet_spaces} that $C_J = O\p{1}$.
		\item In the setting of Section \ref{subsection:examples_GPs}, if $\p{\XX,\mu}$ is either an open, connected and bounded domain of $\RR^d$ with $\CC^\infty-$boundary and $\mu$ the Lebesgue measure, or a compact Dirichlet space satisfying the set of assumptions \ref{assumptions:dirichlet_spaces}, $k$ a reproducing kernel over $\XX$ with RKHS $\Hbb$ that is norm equivalent to the Sobolev space $H^{\alpha + d/2}\p{\XX}$ and $\mu_0$ a probability measure over $\XX$ with a density $p_0 = \frac{d\mu_0}{d\mu}$ with respect to $\mu$ that is upper and lower bounded by positive constants, then for $\p{u_j}_{j \geq 1}$ an orthonormal basis of $L^2\p{\mu_0}$ associated with the kernel operator $f \mapsto \int k(\cdot,y)f(y)\mu_0(dy)$ associated with eigenvalues $\lambda_1 \geq \lambda_2 \geq \ldots$, we have $C_J = O\p{J^\delta}$ for any $\delta > 0$, see the discussion at the end of Section \ref{subsection:general_gps}.
	\end{enumerate}
	
	\begin{remark}\label{remark:suboptimality_laplacian_eigenbasis}
	The constant $C_J$ automatically allows to controls the supremum norm of the $u_j's$ :
	\[
	\norm{u_j}_{L^\infty\p{\mu}} \leq C_j \sqrt{j}.
	\]
	This bound is sharp in the the wavelet basis example $2)$ above, but can be pessimistic in some settings of interest : for instance the optimal scaling when $\p{u_j}_{j \geq 1}$ is an orthonormal basis of the Laplace-Beltrami operator on a compact manifold associated with eigenvalues $\p{\lambda_j}_{j \geq 1}$ ordered descreasingly (for which we know that $C_j = O(1)$ according the point $3)$ above) is known to be $\lambda_j^{\frac{d-1}{4}} \asymp j^{\frac{1}{2} - \frac{1}{2d}} = o\p{\sqrt{j}}$ (where we have used the asymptotic estimate $\lambda_j \asymp j^{2/d}$ \ref{proposition:growth_eigenpairs_dirichlet_spaces}), see e.g. \cite{donnellyEigenfunctionsLaplacianCompact2006}. Nevertheless, this suboptimality will not be an obstruction to derive optimal contraction rates, see the results in Section \ref{subsection:hierarchical_GPs_dirichlet_spaces} below.
	\end{remark}
	
	We now consider a probability measure $\mu_0$ on $\XX$ and $\p{X_i}_{i=1}^n \iid \mu_0$. We let $\Sigma_n \in \RR^{J \times J}$ be the empirical inner product matrix of $\p{u_1,\ldots,u_J}$ with respect to the design $\p{X_i}_{i=1}^n$
	\[
	\p{\Sigma_n}_{jk} := \Pro_n \c{u_j u_k} = \frac{1}{n} \sum_{i=1}^n u_j\p{X_i} u_k\p{X_i},
	\]
	as well as $\Sigma$ its population counterpart
	\[
	\Sigma_{jk} = \int u_j(x) u_k(x) p_0(x) \mu(dx).
	\]
	We assume that $\Sigma$ satisfies the following condition
	\begin{equation}\label{eqn:assumption_mu_0_u_j}
	p_- \norm{v}_{\RR^J}^2 \leq v^T \Sigma v \leq p_+ \norm{v}_{\RR^J}^2, \quad v \in \RR^J,
	\end{equation}
	for some $p_-,p_+ > 0$, or equivalently that $\Sigma$ has a spectrum contained in $\c{p_-,p_+}$. This is the case for instance if the $u_j's$ are orthonormal in $L^2\p{\mu_0}$ (and in that case $p_- = p_+ = 1$), but also if they are orthonormal in $L^2\p{\mu}$ and $\mu_0$ has a density $p_0$ with respect to $\mu$ that satisfies
	\[
	0 < p_- = \underset{x \in \XX}{\essinf_\mu} p_0 \leq p_+ = \underset{x \in \XX}{\esssup_\mu} p_0 < +\infty.
	\]
	as in that case
	\[
	v^T \Sigma v = \sum_{j,k = 1}^J v_j v_k \int u_j(x) u_k(x) p_0(x) \mu(dx) = \int \p{\sum_{j=1}^J v_j u_j(x)}^2 p_0(x) \mu(dx),
	\]
	which is itself upper (respectively : lower) bounded by $p_+ \norm{v}_{\RR^J}^2$ (respectively : $p_- \norm{v}_{\RR^J}^2$). More generally this also holds under the same assumption on $p_0$ if $\p{u_j}_{j = 1}^J$ is a frame in $L^2\p{\mu}$ with frame bounds independent of $J$. We can now apply the results in Chapter 5 of \cite{Eldar_Kutyniok_2012} to prove concentration of $\Sigma_n$ around $\Sigma$.
	
	\begin{theorem}\label{theorem:1}
		If $\mu_0$ and $\p{u_j}_{j=1}^J$ satisfy \ref{eqn:assumption_mu_0_u_j}, for a universal constant $c_0>0$ such that if $c_0 n \geq C_J^2 J \ln J \norm{\Sigma}^{-1} \p{p_+/p_-}^2$ we have
		\[
		\Pro_0 \p{\lambda_- \leq \lambda_{\min}\p{\Sigma_n} \leq \lambda_{\max}\p{\Sigma_n} \leq \lambda_+} \geq \Pro_0 \p{\norm{\Sigma_n - \Sigma} \leq \frac{p_-}{2p_+} \norm{\Sigma}} \geq 1 - e^{1-c_0 n\p{p_-/p_+}^2 \norm{\Sigma}/C_J^2 J},
		\]
		where $\lambda_- = p_-/2$ and $\lambda_+ = p_+ + p_-/2$.
	\end{theorem}
	\begin{proof}
		Let $u$ be the random variable $u = \c{u_1(x),\ldots,u_J(x)}^T, x \sim \mu_0$ with values in $\RR^J$. Then $\norm{u}_{\RR^J} \leq C_J\sqrt{J}$ almost surely and therefore, for any $\varepsilon \in \p{0,1}$ and $t \geq 0$, Corollary 5.44\footnote{Corollary 5.52 requires $t \geq 1$ which is why we use Corollary 5.44 instead} of Vershynin \cite{Eldar_Kutyniok_2012} (combined with the footnote on page 240) implies that for some universal constant $c_0>0$, with probability at least $1 - J e^{-c_0t^2}$,
		\[
		\norm{\Sigma_n - \Sigma} \leq \max \p{\norm{\Sigma}^{1/2}t\sqrt{\frac{C_J^2 J}{n}}, \frac{C_J^2 J}{n}}.
		\]
		If $n \geq \p{t/\varepsilon}^2 \norm{\Sigma}^{-1} C_J^2 J$ we thus have
		\[
		\norm{\Sigma_n - \Sigma} \leq \varepsilon \norm{\Sigma},
		\]
		using $\varepsilon \in \p{0,1}$.
		Since $\Sigma$ and $\Sigma_n$ are symmetric positive definite this implies that
		\[
		\lambda_{\max}\p{\Sigma_n} = \norm{\Sigma_n} \leq \norm{\Sigma - \Sigma_n} + \norm{\Sigma} \leq \p{1+\varepsilon}\norm{\Sigma} = \p{1+\varepsilon}\lambda_{\max}\p{\Sigma} \leq \p{1+\varepsilon}p_+.
		\]
		For the lower bound on $\lambda_{\min}\p{\Sigma_n}$, we first notice that
		\[
		\norm{I - \Sigma^{-1}\Sigma_n} \leq \norm{\Sigma^{-1}} \norm{\Sigma - \Sigma_n} \leq \norm{\Sigma^{-1}} \norm{\Sigma}\varepsilon = \lambda_{\min}\p{\Sigma}^{-1}\lambda_{\max}\p{\Sigma} \leq p_-^{-1} p_+ \varepsilon. 
		\]
		In particular for $\varepsilon < p_- p_+^{-1}$ this implies that $\Sigma^{-1}\Sigma_n = I - \p{I - \Sigma^{-1}\Sigma_n}$ (and therefore $\Sigma_n$ itself) is invertible and
		\[
		\Sigma_n^{-1} \Sigma = \p{\Sigma^{-1}\Sigma_n}^{-1} = \sum_{l \geq 0} \p{I - \Sigma^{-1}\Sigma_n}^l.
		\]
		Therefore we obtain
		\[
		\Sigma_n^{-1} = \sum_{l \geq 0} \p{I - \Sigma^{-1}\Sigma_n}^l \Sigma^{-1},
		\]
		and taking the norm yields
		\[
		\lambda_{\min}\p{\Sigma_n}^{-1} = \norm{\Sigma_n^{-1}} \leq \sum_{l \geq 0} \norm{I - \Sigma^{-1}\Sigma_n}^l \norm{\Sigma^{-1}} \leq p_-^{-1} \sum_{l \geq 0} \p{\varepsilon p_-^{-1} p_+}^l = \frac{p_-^{-1}}{1 - \varepsilon p_-^{-1} p_+}.
		\]
		We now conclude by taking $\varepsilon = \frac{p_-}{2p_+} \in \p{0,1}$ and $t^2$ such that $c_0 t^2 \geq 2\ln J$ and $n \geq \p{t/\varepsilon}^2 \norm{\Sigma}^{-1} C_J^2 J$ : it suffices to choose $t = \frac{n\varepsilon^2\norm{\Sigma}}{C_J^2 J}$, which is satisfies $t^2 \geq 2\ln J/c_0$ by assumption (perhaps for another universal constant $c_0$).
	\end{proof}
	
	Notice that for any function $f = \sum_{j = 1}^J \alpha_j u_j \in V_J := \Span \b{u_j : 1 \leq j \leq J}$ we have
	\[
	p_- \norm{f}_2^2 = p_- \alpha^T \alpha \leq p_- \lambda_-^{-1} \alpha^T \Sigma_n \alpha = 2 \norm{f}_n^2 \iff \norm{f}_2 \leq \sqrt{\frac{2}{p_-}} \norm{f}_n,
	\]
	and therefore
	\[
	\lambda_- \leq \lambda_{\min}\p{\Sigma_n} \leq \lambda_{\max}\p{\Sigma_n} \leq \lambda_+ \implies \forall f \in V_J, \quad \norm{f}_2 \leq \sqrt{\frac{2}{p_-}} \norm{f}_n.
	\]
	Thus, Theorem \ref{theorem:1} allows to bound the $L^2$ norms of functions in $V_J$ by their empirical $L^2$ norms on a high probability event. As a consequence, Theorem \ref{theorem:1} cann allow to extend preliminary $\norm{\cdot}_n-$posterior contraction results to $\norm{\cdot}_2-$ ones with the same rates, provided that, for some $J_n$ that satisfies $C_J^2 J \ln J \lesssim n$, the orthonormal projection of $f-f_0$ on the orthogonal complement of $V_{J_n}$ in $L^2\p{\XX,\mu}$ is sufficiently small on a set of high posterior probability. To this end, for $f : \XX \to \RR$ and $J \geq 1$, define the best uniform approximation error $\Ecal_J\p{f}$ of $f$ by $V_J$ as
	\[
	\Ecal_J\p{f}_\infty := \inf_{h \in V_J} \norm{f-h}_\infty.
	\]
	We then have the following result.
	\begin{theorem}\label{theorem:2}
		If $\mu_0$ and $\p{u_j}_{j=1}^J$ satisfy \ref{eqn:assumption_mu_0_u_j} and for an index $J_n$ satisfying $n \geq c_0 C_{J_n}^2 J_n \ln J_n \norm{\Sigma}^{-1} \p{p_+/p_-}^2$ (where $c_0$ is the constant from Theorem \ref{theorem:1}) we have
		\begin{enumerate}
			\item[1)] $\EE_0 \Pi \c{\norm{f-f_0}_n > \varepsilon_n | \Xbb^n} \xrightarrow[n \to \infty]{} 0$,
			\item[2)] $\Ecal_{J_n}\p{f_0}_\infty < M_1\varepsilon_n$,
			\item[3)] $\EE_0 \Pi \c{\Ecal_{J_n}\p{f}_\infty \geq M_2\varepsilon_n | \Xbb^n} \xrightarrow[n \to \infty]{} 0$,
		\end{enumerate}
		for some $M_1,M_2 > 0$, then for some $M>0$,
		\[
		\EE_0 \Pi \c{\norm{f-f_0}_2 > M\varepsilon_n | \Xbb^n} \xrightarrow[n \to \infty]{} 0.
		\]
	\end{theorem}
	\begin{proof}
		The proof is a trivial application of the triangle inequality. Let $A \subset \RR^n$ be the event
		\[
		A = \b{x \in \RR^n : \lambda_- \leq \lambda_{\min}\p{\Sigma_n} \leq \lambda_{\max}\p{\Sigma_n} \leq \lambda_+}, 
		\]
		where $\lambda_- = p_-/2, \quad \lambda_+ = p_+ + p_-/2$. We have for any $M > 0$
		\begin{align*}
			& \EE_0 \Pi \c{\norm{f-f_0}_2 > M\varepsilon_n | \Xbb^n} \\
			& \leq \Pro_0\p{A^c} + \EE_0 \Pi \c{\norm{f-f_0}_n > \varepsilon_n | \Xbb^n} + \EE_0 \Pi \c{\Ecal_{J_n}\p{f}_\infty \geq M_2\varepsilon_n | \Xbb^n} \\
			& + \EE_0 \c{\ind{A}\Pi \c{\Ecal_{J_n}\p{f}_\infty < M_2\varepsilon_n, \norm{f-f_0}_n \leq \varepsilon_n, \norm{f-f_0}_2 > M\varepsilon_n | \Xbb^n}}.
		\end{align*}
		The first three terms are $o(1)$ by $1)$ \& $3)$ and Theorem \ref{theorem:1}. Moreover we have $\Ecal_{J_n}\p{f-f_0}_\infty \leq \Ecal_{J_n}\p{f}_\infty + \Ecal_{J_n}\p{f_0}_\infty$, so that the set
		\[
		\b{f : \Ecal_{J_n}\p{f}_\infty < \varepsilon_n, \norm{f-f_0}_n \leq \varepsilon_n, \norm{f-f_0}_2 > M\varepsilon_n}
		\]
		is actually empty for large $M>0$, which concludes the proof.
	\end{proof}
	
	Let us comment on the assumptions of Theorem \ref{theorem:2} before applying it to concrete examples. Assumption $1)$ is a contraction rate result in the $\norm{\cdot}_n$ semi metric, classically obtained with the testing approach under a smoothness type condition that usually implies $\Ecal_{J_n}\p{f_0}_\infty \lesssim \varepsilon_n$. Therefore, conditions $1)$ \& $2)$ typically come for free if a preliminary contraction rate with respect to the $\norm{\cdot}_n$ semi metric has been obtained. The true assumption is $3)$, which is satisfied for instance if the prior puts all of its mass on $V_{J_n}$, or more generally if the prior probability of $f$ having an approximation error $\Ecal_{J_n}\p{f}_\infty$ larger than $\varepsilon_n$ is sufficiently small, as we will see in Sections \ref{subsection:sieve_priors}, \ref{subsection:hierarchical_GPs_dirichlet_spaces} \& \ref{subsection:extension_rggs}.
	
	At this stage it is natural to wonder if Theorem \ref{theorem:2} could be applied in order to show the results of Section \ref{sec:gaussian_processes}. Actually applying Theorem \ref{theorem:2} to these Gaussian process priors is not straightforward even in favourable situations when the kernel has an explicit Mercer expansion in a reference basis (e.g. : Fourier or wavelets), and the main obstacle for this is Assumption $3)$. Indeed, in the setting of Sections \ref{subsection:examples_GPs} \& \ref{subsection:dirichlet_spaces} consider the $1-$dimensional torus $\XX$ equipped with its Riemannian volume form $\mu$ and its Laplace-Beltrami operator $L$ as a compact Dirichlet space, as well as $\p{u_j}_{j \geq 1}$ an $L^2\p{\mu}-$orthonormal basis of eigenfunctions associated with eigenvalues $\lambda_1 \leq \lambda_2 \leq \ldots$ (here $Lf = f''$ in angular coordinates so that $\p{u_j}_{j \geq 1}$ is simply the Fourier basis and $\lambda_j \asymp j^2$). For the probability distribution of the covariates consider for simplicity the case when $\mu_0 = \mu$. If we try to apply Theorem \ref{theorem:2} with a regression function $f_0 \in B_{\infty \infty}^s\p{\XX}$ for some $s>0$ and $\varepsilon_n \asymp n^{-\frac{s}{2s+1}}$ for a Gaussian process prior over $\XX$ with RKHS $H^{s+1/2}\p{\XX}$ and reproducing kernel
	\[
	k(x,y) = \sum_{j \geq 1} \p{1 + \lambda_j}^{-(s+1/2)} u_j(x) u_j(y),
	\]
	i.e.
	\[
	f = \sum_{j \geq 1} \p{1 + \lambda_j}^{-\frac{s+1/2}{2}} Z_j u_j, \quad Z_j \iid \NN\p{0,1}.
	\]
	Using Proposition \ref{prop:characterisation_besov_spaces} and an application of the testing approach we can prove that that $\varepsilon_n$ is a $\norm{\cdot}_n-$posterior contraction rates at $f_0$. In order to conclude that it is also an $L^2\p{\mu}-$posterior contraction rate at $f_0$ using Theorem \ref{theorem:2} we would then have to show that $\EE_0 \Pi \c{\Ecal_{J_n}\p{f}_\infty > M \varepsilon_n | \Xbb^n} \to 0$ for some $M>0$ and where $n^{\frac{1}{2s+1}} \asymp \varepsilon_n^{-1/s} \lesssim J_n$ is such that $c_0\sqrt{2} \norm{\Sigma}^{-1} \p{p_+/p_-}^2 J_n \ln J_n \leq n$ (here $C_J \leq \sqrt{2}$). Without using the explicit representation of the posterior distribution as in \ref{sec:gaussian_processes}, obtaining such a posteriori control typically relies on an application of the remaining mass theorem (see \cite{castilloBayesianNonparametricStatistics2024} Lemma 1.2), i.e. we aim to show that for any $M,c_1 > 0$ we can find $c_2>0$ such that $\Pi \c{\Ecal_{c_2 J_n}\p{f}_\infty > M \varepsilon_n} \leq e^{-c_1 n \varepsilon_n^2}$. Obtaining such estimate is possible using a simple triangle inequality $\Ecal_J(f)_\infty \lesssim \sum_{j > J} j^{-(s+1/2)} \abs{Z_j} \norm{u_j}_\infty \lesssim \sum_{j > J} j^{-(s+1/2)} \abs{Z_j}$, and a Chernoff bound yields, for all $\lambda > 0$ and some $c>0$,
	\[
	\Pi \c{\Ecal_J(f)_\infty > \varepsilon} \leq \exp\p{-\lambda \varepsilon + cJ^{-2s} \lambda^2},
	\]
	which implies, with $\lambda = \varepsilon J^{2s}/2c$,
	\[
	\Pi \c{\Ecal_J(f)_\infty > \varepsilon} \leq \exp \p{-\varepsilon^2 J^{2s}/4c}.
	\]
	The latter can be made less than $e^{-c_1n\varepsilon_n^2}$ if $J \gtrsim n^{1/2s}$, but $J$ still has to satisfy $J \ln J \lesssim n$, therefore this approach leads to optimal rates only in the regime $s>1/2$. Instead of this simple argument based on the triangle inequality and Chernoff bound, one could also use more involved estimates on the supremum norm of Gaussian process (see e.g. \cite{ghosalFundamentalsNonparametricBayesian2017}), but this still does not seem to allow $s \leq 1/2$. Notice that the problem does not arise if one instead considers the truncated process
	\[
	f = \sum_{j = 1}^J \p{1 + \lambda_j}^{-\frac{s+1/2}{2}} Z_j u_j, \quad Z_j \iid \NN\p{0,1}, \quad n^{\frac{1}{2s+1}} \lesssim J \ll \frac{n}{\ln n},
	\]
	as in this case $\Ecal_J(f)_\infty = 0$ $\Pi-$almost surely (and therefore $\Pi \c{\cdot|\Xbb^n}-$almost surely, $\Pro_0-$almost surely) and $J$ satisfies the assumption of Theorem \ref{theorem:2}. These truncated processes are the object of the next subsection.
	
	\subsection{Sieve priors}\label{subsection:sieve_priors}
	
	We show a first basic example of application of Theorem \ref{theorem:2} in the case of priors based on truncated basis expansions : more precisely, let $\Pi$ be the probability distribution of the random function
	\begin{equation}\label{prior:truncated}
		f = \sum_{j=1}^J Z_j u_j,
	\end{equation}
	where the coefficients $Z_j$ are \iidt sampled according to a probability density $\Psi$ on $\RR$ satisfying the assumption
	\begin{assumption}\label{assumption:3}
		$\Psi$ is positive, continuous and there exists $a_0,c_1,c_2 ,c_3> 0$ such that for any $a \geq a_0$,
		\[
		\int_{\abs{x} > a} \Psi(x)dx \leq c_1 e^{-c_2 a^{c_3}},
		\]
	\end{assumption}
	and $J = J_n \geq 1$ is deterministic and satisfies $n \geq c_0 C^2 J_n \ln J_n \norm{\Sigma}^{-1} \p{p_+/p_-}^2$ (where $c_0$ is the constant from Theorem \ref{theorem:1}) or $J \sim \pi_J$ (random truncation) with, for some $a_1,a_2,b_1,b_2>0$,
	\begin{equation}\label{prior:J}
		\forall j \geq 1, a_1 e^{-b_1 j L_j} \leq \pi_J(j) \leq a_2 e^{-b_2 j L_j}, \quad L_j = 1 \text{ or } \ln j.
	\end{equation}
	In particular, the standard Gaussian density satisfies Assumption \ref{assumption:3}, and in the case of a random truncation $\pi_J$ can be either Poisson ($L_j = \ln j$) or geometric ($L_j = 1$). The role of the truncation $J$ is at a high level to adjust bias and variance of the method, as more basis functions will be needed to precisely estimate rougher regression functions. Priors defined as in Equation \ref{prior:truncated} are commonly referred to as \textit{sieve priors}, see \cite{arbelBayesianOptimalAdaptive2013} and the references therein.
	
	We make the following smoothness assumption on $f_0$ :
	
	\begin{assumption}\label{assumption:4}
		There exists $K,d,s > 0$ such that, for any $J \geq 1$ we have $\Ecal_J\p{f_0}_\infty \leq K J^{-s/d}$.
	\end{assumption}
	
	Assumption \ref{assumption:4} quantifies the quality of approximation of $f_0$ in the basis $\p{u_j}_{j \geq 1}$. For instance, if $\XX = \c{0,1}^d$ and $\p{u_j}_{j \geq 1}$ is an $S-$regular wavelet basis (see Chapter 4  \cite{gineMathematicalFoundationsInfiniteDimensional2015}) and $s < S$, Assumption \ref{assumption:4} holds whenever $f_0$ is $s-$H\"older (or more generally if $f_0$ is an element of the Besov space $B_{\infty \infty}^s\p{\XX}$), and in that case $d$ is simply the dimension of $\XX$. Similarly, for a compact Dirichlet space $\XX$ as in Section \ref{subsection:examples_GPs} that satisfies the set of Assumptions \ref{assumptions:dirichlet_spaces} and for $\p{u_j}_{j \geq 1}$ an orthonormal basis of $L^2\p{\mu}$ associated with the (non increasingly ordered) eigenvalues $\p{\lambda_j}_{j \geq 1}$ of the operator $L$, Proposition \ref{prop:characterisation_besov_spaces} shows that Assumption \ref{assumption:4} holds if and only if $f \in B_{\infty \infty}^s$, where the Besov space $B_{\infty \infty}^s$ is defined in Section \ref{subsection:dirichlet_spaces}. We then obtain the following result as an application of the testing approach and Theorem \ref{theorem:2} :
	
	\begin{theorem}\label{theorem:3}
		For a prior generated as \ref{prior:truncated} and satisfying Assumption \ref{assumption:3} on the density $\Psi$ with either a deterministic or a random truncation (in which case $\pi_J$ satisfies \ref{prior:J}), $\mu_0$ and $\p{u_j}_{j=1}^J$ satisfying \ref{eqn:assumption_mu_0_u_j} and a true regression function satisfying Assumption \ref{assumption:4}, if $C_j = o \p{j^{s/d}}$ we have
		\[
		\EE_0 \Pi \c{\norm{f-f_0}_{L^2\p{\mu_0}} > M \varepsilon_n|\Xbb^n} \to 0,
		\]
		for some $M > 0$ large enough, where
		\[
		\varepsilon_n = \begin{cases}
			J^{-s/d} + \sqrt{\frac{J \ln n}{n}} \text{ in the deterministic $J$ case,} \\
			\p{n/ \ln n}^{-\frac{s}{2s+d}} \text{ in the random $J$ case.}
		\end{cases}
		\]
	\end{theorem}
	\begin{proof}
		The testing approach yields a posterior contraction rate with respect to the $\norm{\cdot}_n$ at rate $\varepsilon_n$, while the remaining mass theorem (see e.g. \cite{castilloBayesianNonparametricStatistics2024} Lemma 1.2) implies that $\EE_0 \Pi \c{\Ecal_{J_n}\p{f}_\infty \leq M\varepsilon_n|\Xbb^n}$ for some $M$, where $C_{J_n}^2 J_n \ln J_n = o\p{n}$. Since in addition $f_0$ satisfies Assumption \ref{assumption:4}, Theorem \ref{theorem:3} is a direct consequence of Theorem \ref{theorem:2}. We defer the full proof to Section \ref{appendix:proof_random_series}.
	\end{proof}
	
	\subsection{Square exponential Gaussian processes on compact Dirichlet spaces}\label{subsection:hierarchical_GPs_dirichlet_spaces}
	
	In this section, we demonstrate through an example how the results of Section \ref{subsection:sieve_priors} can extend to infinite series. More precisely, we show how to extend the results in \cite{castilloThomasBayesWalk2014} to $L^2\p{\mu_0}$ contraction, using the exponential decay on the coefficients induced by the prior. The work in \cite{castilloThomasBayesWalk2014} considers a compact Dirichlet space as in Section \ref{subsection:examples_GPs} that satisfies \ref{assumptions:dirichlet_spaces} and construct a prior $\Pi$ on regression functions $f$ as a hierarchical Gaussian process, where conditional on a random $t \sim \pi_1$ on $\p{0,\infty}$, the distribution of $f$ is a (centered) Gaussian process on $\XX$ with covariance kernel
	\[
	k_t(x,y) := \sum_{j \geq 1} e^{-t\lambda_j} u_j(x)u_j(y),
	\]
	where $\p{u_j}_{j \geq 1}$ is an orthonormal basis of $L^2\p{\mu}$, $\lambda_j$ is the eigenvalue of $\LL$ associated with $u_j$, ordered such that $\lambda_1 \geq \lambda_2 \geq \ldots$. In particular, if $\XX$ is a compact manifold, $\mu$ its Riemannian volume element and $\LL$ is the Laplace-Beltrami operator, $k_t$ is the heat kernel of $\XX$ with respect to $\mu$ (see e.g. \cite{grigoryanHeatKernelAnalysis2013}), but the framework allows more general situations. It is then shown in \cite{castilloThomasBayesWalk2014} that under the nonparametric regression model \ref{model:nonparametric_regression}, for a prior $\pi_1$ on $t$ having a density (still denoted by $\pi_1$) with respect to the Lebesgue measure satisfying
	\begin{equation}\label{eqn:prior_t_castillo}
		c_1 t^{-a} e^{-t^{-d/2} \ln^{1+d/2}\p{1/t}} \leq \pi_1(t) \leq c_2 t^{-a} e^{-t^{-d/2} \ln^{1+d/2}\p{1/t}} \text{ for some $c_1,c_2,a>0$},
	\end{equation}
	a uniformly continuous and bounded regression function $f_0 \in B_{\infty \infty}^s, \quad s > 0$ (see \ref{subsection:dirichlet_spaces} for a definition of the Besov spaces $B_{\infty \infty}^s$), the posterior distribution contracts at rate $\p{n/\ln n}^{-\frac{s}{2s+d}}$ with respect to the empirical $L^2$ norm $\norm{\cdot}_n$. Notice that according to Proposition \ref{prop:characterisation_besov_spaces}, $f_0$ satisfies $\Ecal_J\p{f_0}_\infty \lesssim J^{-s/d}$. Using Theorem \ref{theorem:2} we can actually extend this posterior contraction result to the $L^2\p{\mu_0}$ norm.
	
	\begin{theorem}\label{theorem:L2_norm_castillo}
		In the context of \cite{castilloThomasBayesWalk2014} and assumption \ref{eqn:assumption_mu_0_u_j} on $\mu_0$ (for instance : $\mu_0$ admits a density $\frac{d\mu_0}{d\mu}$ with respect to $\mu$ that is upper and lower bounded by constants), for $M>0$ large enough the posterior distribution satisfies
		\[
		\EE_0 \Pi \c{\norm{f-f_0}_{L^2\p{p_0}} > M \varepsilon_n | \Xbb^n} \to 0, \quad \varepsilon_n = \p{\frac{\ln n}{n}}^{\frac{s}{2s+d}}.
		\]
	\end{theorem}
	\begin{proof}
		The proof is essentially the same as the one of Theorem \ref{theorem:3} : the testing approach yields $\norm{\cdot}_n$ contraction at rate $\varepsilon_n$, the remaining mass theorem implies $\EE_0 \Pi \c{\Ecal_{J_n}\p{f} \leq M\varepsilon_n | \Xbb^n}$ for some $M>0$ and $C_{J_n}^2 J_n \ln J_n = o\p{n}$ thanks to the exponentially decaying prior on the coefficients, so that Theorem \ref{theorem:L2_norm_castillo} is a direct consequence of Theorem \ref{theorem:2}. We defer the full proof to Section \ref{subsection:hierarchical_GPs_dirichlet_spaces}.
	\end{proof}
	
	\begin{remark}
		Although the upper bounds on the posterior contraction rates presented in \cite{castilloThomasBayesWalk2014} are formulated for (hierarchical) Gaussian process priors, these would also hold for generic random series priors as the ones in Section \ref{subsection:sieve_priors}, perhaps up to additional logarithmic factors. As a consequence similar $L^2\p{\mu_0}$ posterior contraction rates for these priors could be derived in the same way by applying Theorem \ref{theorem:2}.
	\end{remark}
	
	\subsection{Extension to priors based on graph Laplacian regularisation}\label{subsection:extension_rggs}
	
	In this section, we show how to derive an analog of Theorem \ref{theorem:3} in the context of graph based semi-supervised learning, where the base space $\XX$ is replaced by a set of random covariates $\XX_N := \b{X_1,\ldots,X_N}, N \geq n$ containing the design $\b{X_1,\ldots,X_n}$ and the basis $\p{u_j}_{j \geq 1}$ is learned from the covariates themselves, see e.g. \cite{greenMinimaxOptimalRegression2021,ficheraImplicitManifoldGaussian2023,dunsonGraphBasedGaussian2022,sanz-alonsoUnlabeledDataHelp2022,rosaNonparametricRegressionRandom2024a}. More precisely, we adopt the setting considered in \cite{rosaNonparametricRegressionRandom2024a} : $\MM \subset \RR^D$ is a compact, connected, $\alpha-$H\"older submanifold of $\RR^D$, and $X_1,\ldots,X_N$ are \iidt samples of a probability distribution $\mu_0$ having a density $p_0$ with respect to the Riemannian volume element $\mu$ of $\MM$ satisfying $p_- \leq p_0 \leq p_+$ for some $p_-,p_+>0$. We assume a nonparametric regression model
	\begin{equation}\label{eqn"nonparametric_regression_rggs}
		Y_i = f_0\p{X_i} + \varepsilon_i, \quad 1 \leq i \leq n \text{ for some } f_0 : \MM \to \RR,
	\end{equation}
	where importantly only $n \leq N$ of the covariates are labeled with $Y_i's$, and where $N$ satisfies $N \leq n^B$ for a finite (but arbitrarily large) $B>0$. We then define a random geometric graph with vertices $\XX_N $ by deciding $x$ and $y$ to be neighbours (and we write $x \sim y$) if and only if $\norm{x-y}_{\RR^D} < h$ for some $h>0$ (the graph then obviously depends on the chosen bandwidth $h$). We consider the normalised graph Laplacian operator $\LL$ defined by
	\[
	\p{\LL f}(x) := \frac{1}{h^2 \mu_x} \sum_{y \sim x} \p{f(x) - f(y)},
	\]
	for any $f : \XX_N \to \RR$ and $x \in \XX_N$. $\LL$ is a self-adjoint operator on $L^2\p{\nu}$, where $\nu$ is the normalised degree measure on $\XX_N$, i.e. $\nu_y := \# \b{x \in \XX_N : x \sim y}/\sum_y \# \b{x \in \XX_N : x \sim y}$. Since $\LL$ is self-adjoint in the $N-$dimensional Hilbert space $L^2\p{\nu}$, it can be diagonalised in an orthonormal basis $\p{u_j}_{j=1}^N$ with associated eigenvalues $\lambda_0 \leq \lambda_1 \leq \ldots \leq \lambda_N$. Notice that the eigenpairs depend on the covariates $\XX_N$ and the choice of $h$, but we do not write this explicitly for ease of notations. The prior $\Pi$ is then the probability distribution of the random serie
	\begin{equation}\label{prior:rggs}
		f = \sum_{j=1}^J Z_j u_j,
	\end{equation}
	where $\p{h,J}$ is drawn according to some prior distribution $\pi_1$ and, given $\p{h,J}$, the coefficients $Z_j$ are sampled \iidt according to a probability density $\Psi$ on $\RR$ satisfying Assumption \ref{assumption:3}. More specifically, for the choice of $\pi_1$ we set $h = h_J := \frac{J^{-1/d}}{\ln^{\tau/d} n}$ for some $\tau > 0$ and for the prior on $J$ we either choose it deterministic at some value $J_n \in \b{1,\ldots,N}$ (i.e. fixed hyperparameter), or we set $J \sim \pi_J$ with $\pi_J$ satisfies Assumption \ref{prior:J}. As in Sections \ref{subsection:sieve_priors} \& \ref{subsection:hierarchical_GPs_dirichlet_spaces} and \cite{rosaNonparametricRegressionRandom2024a}, choosing a randomised prior on $\p{h,J}$ is important in order to achieve adaptive posterior contraction rates with respect to the regularity $s$ of $f_0$ (up to logarithmic factors). mportantly geometric (resp. Poisson) prior distributions on $J$ satisfy \ref{prior:J} with $L_j = 1$ (resp. $L_j = \ln j$). While more general choices of priors are considered in \cite{rosaNonparametricRegressionRandom2024a} (in particular non deterministic priors for $h$ given $J$), we abstain from such refinements in the present paper even though our proof technique could be adapted in a similar way.
	
	Although the situation is very similar in spirit to the one considered in Section \ref{subsection:sieve_priors}, a crucial difference is that the covariates $\b{X_1,\ldots,X_n}$ (corresponding to points where an observation $Y_i$ is available) are not sampled i.i.d from $\b{X_1,\ldots,X_N}$, and we must therefore find an alternative proof technique.
	
	We recall the result obtained in \cite{rosaNonparametricRegressionRandom2024a}, which provides sufficient conditions to derive $\norm{\cdot}_n-$posterior contraction rates in this setting (although it is formulated there for more general priors).
	\begin{theorem}{Theorems 3.1 \& 3.4 in \cite{rosaNonparametricRegressionRandom2024a}}\label{theorem:rggs_d_n}\\
		If $f_0 : \MM \to \RR$ is $s-$H\"older for some $s>0$,
		\[
		\begin{cases}
			\ln^\kappa N \leq J \ll \frac{N}{\ln^{1+\tau} n}, \quad \tau > 3d/2, \quad \kappa > d \text{ in the fixed $J$ case} \\
			\tau > d/2 \text{ in the random $J$ case}
		\end{cases}
		\]
		and there exists $B>0$ such that $N \leq n^B$, then for $M>0$ large enough
		\[
		\EE_0 \Pi \c{\norm{f-f_0}_n > M\varepsilon_n | \Xbb^n} \to 0,
		\]
		where $\norm{f-f_0}_n^2 := \sum_{i=1}^n \p{f(x_i) - f_0(x_i)}^2$ and
		\[
		\varepsilon_n = \begin{cases}
			\p{\sqrt{\frac{J}{n}} + J^{-s/d} + \ind{s>1} \frac{1}{\sqrt{NJ^{1/d}}}}\ln^\gamma n \text{ in the fixed $J$ case}, \\
			n^{-\frac{s}{2s+d}}\ln^\gamma n \text{ in the randomised $J$ case},
		\end{cases}
		\]
		for some $\gamma>0$ depending on $s,d,\kappa$ and the prior hyperparameter $\tau$.
	\end{theorem}
	
	Furthermore, under the additional condition $s > d/2$ and other requirements on the prior, a posterior contraction result with respect to the $\norm{\cdot}_N$ norm was obtained in \cite{rosaNonparametricRegressionRandom2024a} using the scalar Bernstein concentration inequality, following the approach of \cite{vaartInformationRatesNonparametric2011}. In analogy with the results obtained in Section \ref{subsection:sieve_priors}, the next Theorem shows that the condition $s > d/2$ is actually not needed in order to get a $\norm{\cdot}_N$ contraction result.
	
	\begin{theorem}\label{theorem:rggs_d_N}
		If $f_0 : \MM \to \RR$ is $s-$H\"older for some $s>0$,
		\[
		\begin{cases}
			\ln^\kappa N \leq J \lesssim \frac{n}{\ln^{1+\tau} n}, \quad \tau > 3d/2, \quad \kappa > d \text{ in the fixed $J$ case}, \\
			\tau > d/2 \text{ in the random $J$ case},
		\end{cases}
		\]
		and there exists $B>0$ such that $N \leq n^B$, then for $M>0$ large enough
		\[
		\EE_0 \Pi \c{\norm{f-f_0}_N > M \varepsilon_n | \Xbb^n} \to 0
		\]
		for $M>0$ large enough, where $\Xbb^n = \b{\p{X_i,Y_i}_{i=1}^n, \p{X_i}_{i > n}}, \quad \varepsilon_n = n^{-\frac{s}{2s+d}} \ln^\gamma n$ and $\gamma>0$ depends on $s,d$ and the prior hyperparameters.
	\end{theorem}
	\begin{proof}
		The proof is at a high level similar to the ones of Theorems \ref{theorem:3} \& \ref{theorem:L2_norm_castillo}. Indeed, results in \cite{rosaNonparametricRegressionRandom2024a} show a contraction of order $\varepsilon_n$ with respect to $\norm{\cdot}_n$ as well as the analog of Assumption \ref{assumption:4} in this context, and a similar remaining mass argument shows that $\EE_0 \Pi \c{J \leq J_n |\Xbb^n} \to 1$ for some $J_n \ll n^{1+3d/2}/\ln n$ (this is a trivial fact in the case of a deterministic prior on $J$). To conclude in the spirit of the proof of Theorems \ref{theorem:2} \& \ref{theorem:3}, it suffices to show that we can apply a concentration inequality similar as the ones obtained in Theorem \ref{theorem:1} in order to relate $\norm{\cdot}_n$ and $\norm{\cdot}_N$. In order to do so we first show using the results in \cite{rosaNonparametricRegressionRandom2024a} that the basis $\p{u_j}_{j=1}^N$ satisfies $C_J = O\p{\ln^{3d/2} n}$ in this context, and, by a symmetrisation argument, that the covariates $\b{X_1,\ldots,X_n}$ can be considered as sampled uniformly without replacement from $\b{x_1,\ldots,x_N}$, allowing us to use the matrix Bernstein inequality in \cite{grossNoteSamplingReplacing2010} instead of the one in \cite{Eldar_Kutyniok_2012}. We defer the full proof to the Appendix.
	\end{proof}
	
	\section{Discussion}\label{section:discussion}
	
	In this work, we have shown that proving optimal $L^2$ posterior contraction rates in the nonparametric regression model with random design was possible through relatively simple arguments and without unnecessary boundedness or smothness requirements on $f_0$. For random series priors applying the matrix Bernstein concentration inequality does indeed lead to optimal $L^2$ posterior contraction rates provided that the priors sufficiently downweights the tail of the basis coefficients, and we have demonstrated that this is typically the case for a large class of priors encountered in practice, with concrete examples including sieve priors or hierarchical Gaussian processes on compact Dirichlet spaces. For Gaussian processes with Sobolev type RKHS we prove our results by taking advantage of the explicit representation of the posterior (which is available thanks to conjugacy), and by using available results on the convergence properties of kernel ridge regression estimators. As byproducts of the analysis we obtain new (to the best of the author's knownledge) upper bounds on the supremum norm of Mercer eigenfunctions of the corresponding kernels which can be of independent interest.
	
	There are several ways in which the work presented here could be extended. First, the results of Sections \ref{sec:gaussian_processes} \& \ref{subsection:hierarchical_GPs_dirichlet_spaces} could be adapted to square exponential Gaussian process priors on domains of $\RR^d$. This would require control of the associated Mercer eigenvalues and eigenfunctions as well as an extension of the approach in \cite{fischerSobolevNormLearning2020a} to this setting. Moreover, the results on kernel ridge regression estimators obtained in \cite{fischerSobolevNormLearning2020a} are valid not only under the $L^2\p{\mu_0}$ metric, but also under Sobolev type ones, and it should be possible to extend these to Gaussian processes with Sobolev type RKHS by a refinement of the proof of Theorem \ref{thm:general_thm_GPs}. Finally, Theorem \ref{theorem:2} crucially relies on the Assumption $\EE_0 \Pi \c{\Ecal_{J_n}(f)_\infty \geq M_2 \varepsilon_n | \Xbb^n} \xrightarrow[n \to \infty]{} 0$ which although as we saw is satisfied in various cases of interest is not easily seen to hold for some priors (e.g. : the ones considered in Section \ref{sec:gaussian_processes}, as discussed at the end of Section \ref{subsec:general_results_random_series}), and finding ways to relax it would be particularly valuable.
	
	\section{Acknowledgments}
	The author is grateful to Judith Rousseau for helpful discussions. This work was supported by the Engineering and Physical Sciences Research Council [Grant Ref: EP/Y028732/1].
	
	\printbibliography
	
	\section{Appendix}
	
	\subsection{Some results on compact Dirichlet spaces}\label{subsection:dirichlet_spaces}
	
	In this section we consider a compact metric space $\XX$ satisfying \ref{assumptions:dirichlet_spaces}. The set of assumptions \ref{assumptions:dirichlet_spaces} implies useful estimates on the spectrum and the eigenfunctions of $L$, as shown in the following proposition.
	
	\begin{proposition}\label{proposition:growth_eigenpairs_dirichlet_spaces}
	Under the set of assumptions \ref{assumptions:dirichlet_spaces} the eigenvalues $\p{\lambda_j}_{j \geq 1}$ of $L$ satisfy $\lambda_j \asymp j^{2/d}$ as $j \to \infty$ and the associated eigenfunctions $\p{u_j}_{j \geq 1}$ satisfy
	\[
	\sup_{J \geq 1} \norm{\frac{1}{J} \sum_{j \leq J} u_j^2}_{\infty} < + \infty.
	\]
	\end{proposition}
	\begin{proof}
	The eigenvalues growth estimate follows directly by Proposition $3.20$ in \cite{coulhonHeatKernelGenerated2012} : indeed, Equation $3.50$ of the same paper together with the volume regularity of the the set of assumptions \ref{assumptions:dirichlet_spaces} imply
	\[
	c\Lambda^{d/2} \leq \# \b{j : \lambda_j \leq \Lambda} \leq C \Lambda^{d/2}
	\]
	for some constants $C,c>0$ and for sufficiently large $j \geq 1$. In particular $j > C \Lambda^{d/2} \iff \Lambda < \p{j/C}^{2/d}$ implies $\lambda_j > \Lambda$, which holds in particular for $\Lambda = \p{j/2C}^{2/d}$, thus giving the lower bound $\lambda_j \geq \p{j/2C}^{2/d}$. Similarly, setting $\Lambda = \p{j/c}^{2/d} \iff j = c \Lambda^{d/2}$ yields $j \leq \# \b{j : \lambda_j \leq \Lambda}$, and therefore $\lambda_j \leq \Lambda = c \Lambda^{d/2}$.
	
	For the bound on the eigenfunctions we use the heat kernel bound contained in the set of assumptions \ref{assumptions:dirichlet_spaces} : in particular on the diagonal we have
	\[
	\sup_{x \in \XX} p_t(x,x) \leq \frac{C}{\mu \p{B(x,\sqrt{t})}}, \quad 0 < t \leq 1,
	\]
	which coupled with the volume regularity assumption implies
	\[
	\sup_{x \in \XX} p_t(x,x) \leq \frac{C}{ct^{d/2}}.
	\]
	But now, for any $\Lambda \geq 1$,
	\begin{align*}
	\sum_{\lambda_j \leq \Lambda} u_j^2(x) & \leq \sum_{j \geq 1} e^{1-\lambda_j/\Lambda} u_j^2(x) \\
	& = e p_{\Lambda^{-1}}(x,x) \\
	& \leq \frac{eC}{c} \Lambda^{d/2},
	\end{align*}
	therefore with $\Lambda = \p{J/c}^{2/d} \iff J = c\Lambda^{d/2}$ we get $\lambda_j \leq \Lambda$ whenever $j \leq J = c\Lambda^{d/2}$ by the estimates on $\lambda_j$ obtained above, and therefore
	\[
	\sum_{j \leq J} u_j^2(x) \leq \sum_{\lambda_j \leq \Lambda} u_j^2(x) \leq \frac{eC}{c} \Lambda^{d/2} = \frac{eC}{c^2}J,
	\]
	which finally yields
	\[
	\sup_{J \geq 1} \norm{\frac{1}{J} \sum_{j \leq J} u_j^2}_{\infty} < + \infty.
	\]
	\end{proof}
	
	Following \cite{coulhonHeatKernelGenerated2012}, the eigendecomposition of $L$ allows us to define the Besov spaces $B_{pq}^s, \quad s > 0, \quad 1 \leq p \leq \infty, \quad 0 < q \leq \infty$ as the (quasi) Banach space (Banach space when $q \geq 1$) in the following way : take any $\varphi_0,\varphi \in \CC^\infty\p{\RR^+,\RR}$ such that
	\begin{itemize}
		\item $\supp \varphi_0 \subset \c{0,2}, \quad \varphi_0^{(\nu)} (0) = 0$ for $\nu \geq 1$ and $\abs{\varphi_0(\lambda)} \geq c > 0$ for $\lambda \in \c{0,2^{3/4}}$,
		\item $\supp \varphi \subset \c{1/2,2}, \quad \abs{\varphi(\lambda)} \geq c > 0$ for $\lambda \in \c{2^{-3/4},2^{3/4}}$,
	\end{itemize}
	define $\varphi_j = \varphi\p{2^{-j} \cdot }$ for $j \geq 1$ and consider the space $B_{pq}^s$ as the subspace of $L^p := L^p\p{\XX,\mu}$ with elements $f$ satisfying
	\[
	\norm{f}_{B_{pq}^s} := \p{\sum_{j \geq 0} \p{ 2^{sj} \norm{\varphi_j\p{\sqrt{L}}f}_{L^p} }^q }^{1/q} < \infty
	\]
	with the usual sup-norm modification when $q = \infty$, where $\varphi_j\p{\sqrt{L}}$ is the integral operator with kernel $\varphi_j\p{\sqrt{L}}(\cdot,\cdot)$ (with respect to $\mu$) given by
	\[
	\varphi_j\p{\sqrt{L}}(x,y) = \sum_{j \geq 1} \varphi_j\p{\sqrt{\lambda_j}} u_j(x) u_j(y), \quad x,y \in \XX.
	\]
	As in \cite{coulhonHeatKernelGenerated2012}, in the case $p = \infty$ the space $L^p = L^\infty$ is instead defined as the Banach space of uniformly continuous and bounded functions equipped with the supremum norm. The definition of $B_{pq}^s$ does not depend on the initial choices of $\varphi_0,\varphi$ up to norm equivalence (we refer to \cite{coulhonHeatKernelGenerated2012} for more details on the construction and properties of functional calculus on $L$). In the cases $\p{p,q} = \p{2,2}$ and $\p{p,q} = \p{\infty,\infty}$ it turns out that useful equivalent characterisations of the resulting Besov spaces are available, as shown in the following Proposition.
	
	\begin{proposition}\label{prop:characterisation_besov_spaces}
		For any $s>0$ we have :
	\begin{enumerate}
		\item $\norm{f}_{B_{22}^s} \asymp \norm{f}_{H^s\p{\XX}}$,
		\item $\norm{f}_{B_{\infty \infty}^s} \asymp \norm{f}_\infty + \sup_{J \geq 1} J^{s/d} \Ecal_J\p{f}_\infty$ where $\Ecal_J\p{f}_\infty := \inf_{g \in V_J} \norm{f-g}_\infty, \quad V_J := \Span \b{u_j : 1\leq j \leq J}$.
	\end{enumerate}
	\end{proposition}
	\begin{proof}
	The first statement of the proposition is proved in Proposition $15$ in \cite{rosaPosteriorContractionRates2023} (it was stated there for compact manifolds, but the same proof goes through in the setting of compact Dirichlet spaces satisfying the set of assumptions \ref{assumptions:dirichlet_spaces} as well). For the second statement, with $V^\Lambda = \Span \b{u_j : \lambda_j \leq \Lambda}$ we have, using Proposition $6.2$ and Section $3.5$ of \cite{coulhonHeatKernelGenerated2012},
	\[
	\norm{f}_{B_{\infty \infty}^s} \asymp \norm{f}_\infty + \sup_{\Lambda \geq 1} \Lambda^{s/2} \Ecal^\Lambda\p{f}_\infty,
	\]
	where
	\[
	\Ecal^\Lambda\p{f}_\infty := \inf_{g \in V^\Lambda} \norm{f-g}_\infty.
	\]
	Now for any $J \geq 1$ and $\Lambda \in \c{\lambda_J, \lambda_{J+1}}$ we have, since $\p{\lambda_j}_{j \geq 1}$ is non-decreasing,
	\[
	\Ecal^{\lambda_{J+1}}\p{f}_\infty \leq \Ecal^\Lambda\p{f}_\infty \leq \Ecal^{\lambda_J}\p{f}_\infty,
	\]
	so that, using Proposition \ref{proposition:growth_eigenpairs_dirichlet_spaces},
	\[
	\Lambda^{s/2} \Ecal^\Lambda\p{f}_\infty \leq \Lambda_{J+1}^{s/2} \Ecal^{\lambda_J}\p{f}_\infty \lesssim \p{J+1}^{s/d} \Ecal^{\lambda_J}\p{f}_\infty,
	\]
	But with $V_J := \Span \b{u_j : 1 \leq j \leq J}$ and
	\[
	\Ecal_J\p{f}_\infty := \inf_{g \in V_J} \norm{f-g}_\infty
	\]
	we have $\Ecal^{\lambda_J}\p{f}_\infty \leq \Ecal_J\p{f}_\infty$ because $\lambda_j \leq \lambda_J$ whenever $j \leq J$, therefore
	\[
	\Lambda^{s/2} \Ecal^\Lambda\p{f}_\infty \lesssim J^{s/d} \Ecal_J\p{f}_\infty, \quad J \geq 1,
	\]
	which finally implies
	\[
	\sup_{\Lambda \geq 1} \Lambda^{s/2} \Ecal^\Lambda\p{f}_\infty \lesssim \sup_{J \geq 1} J^{s/d} \Ecal_J\p{f}_\infty.
	\]
	Similarly using Proposition \ref{proposition:growth_eigenpairs_dirichlet_spaces} there exists $c,j_0>0$ such that $\lambda_j \geq cj^{2/d}$ for all $j \geq j_0$ large enough. In particular, $\lambda_j \leq cJ^{2/d}/2$ necessarily implies $j < J$ and therefore $V^{cJ^{2/d}/2} \subset V_J$ for $J \geq j_0$. This implies, perahps for a larger $j_0$,
	\[
	\sup_{\Lambda \geq 1} \Lambda^{s/2} \Ecal^\Lambda\p{f}_\infty \geq \sup_{J \geq j_0} \sup_{cJ^{2/d}/2 \leq \Lambda \leq c\p{J+1}^{2/d}/2} \p{cJ^{2/d}/2}^{s/2} \Ecal^{c\p{J+1}^{2/d}/2}\p{f}_\infty \gtrsim \sup_{J \geq j_0} \p{J+1}^{s/d}\Ecal_{J+1}\p{f}_\infty.
	\]
	In particular
	\[
	\sup_{\Lambda \geq 1} \Lambda^{s/2} \Ecal^\Lambda\p{f}_\infty \gtrsim \sup_{J \geq j_0 + 1} J^{s/d} \Ecal_J\p{f}_\infty.
	\]
	But for any $J \leq j_0$ we also have
	\[
	\Ecal_J(f)_\infty = \inf_{g \in V_J} \norm{f-g}_\infty \leq \norm{f}_\infty \implies \norm{f}_\infty + \sup_{J \geq 1} J^{s/d} \Ecal_J\p{f}_\infty \leq 2\norm{f}_\infty + \sup_{J \geq j_0+1} J^{s/d} \Ecal_J\p{f}_\infty,
	\]
	and therefore
	\[
	\norm{f}_\infty + \sup_{J \geq 1} J^{s/d} \Ecal_J\p{f}_\infty \lesssim \norm{f}_\infty + \sup_{\Lambda \geq 1} \Lambda^{s/2} \Ecal^\Lambda\p{f}_\infty,
	\]
	thus concluding the proof.
	\end{proof}
	
	Using the set of assumptions \ref{assumptions:dirichlet_spaces} we now show two embeddings theorems satisfied by the Sobolev spaces $H^{\alpha+d/2}\p{\XX}, \alpha > 0$.
	
	\begin{proposition}\label{prop:embeddings}
		For any $\alpha > 0$ we have
		\begin{enumerate}
			\item $H^{\alpha+d/2}\p{\XX}$ is continuously embedded in $\CC\p{\XX}$,
			\item $H^{\alpha+d/2}\p{\XX}$ is compactly embedded in $L^2$.
		\end{enumerate}
	\end{proposition}
	\begin{proof}
		\begin{itemize}
			\item Let $f \in H^{\alpha+d/2}\p{\XX} \subset L^2$ ; we will show that $f$ is equal $\mu-$almost everywhere to a continuous function $\tilde f$ over $\XX$ satisfying $\norm{\tilde f}_{L^\infty\p{\mu}} = \sup_{x \in \XX} \abs{\tilde f(x)} \lesssim \norm{f}_{H^{\alpha+d/2}\p{\XX}}$. We first notice that $u_j \in \CC\p{\XX}$ for any $j \geq 1$ : indeed, for any $t \in \p{0,t_1}$,
			\[
			e^{-t_1\LL}u_j = e^{-t_1\lambda_j} u_j = \int p_1(\cdot,y) u_j(y) \mu(dy) \iff u_j = e^{t_1\lambda_j} \int p_1(\cdot,y) u_j(y) \mu(dy),
			\]
			and the continuity of the function in the right hand side of the last display now follows from standard integration theory using the domination $\abs{p_t(\cdot,y) u_j(y)} \lesssim t^{-d/2} \abs{u_j(y)}$ (which is integrable in $y$ since $u_j \in L^2\p{\mu}$ and $\mu\p{\XX} < \infty$) and the H\"older continuity from the set of assumptions \ref{assumptions:dirichlet_spaces}. Now consider the sequence of functions
			\[
			S_q := \sum_{j = 1}^q \inner{u_j|f}_{L^2\p{\mu}} u_j \in \CC\p{\XX}, \quad q \geq 1.
			\]
			Then $\p{S_q}_{q \geq 1}$ is a Cauchy sequence in $\CC\p{\XX}$ equiped with the supremum norm : indeed by Cauchy-Schwarz inequality we have, for any $q \geq p \geq 1$ and $x \in \XX$,
			\begin{align*}
				\abs{S_q(x) - S_p(x)} & = \abs{\sum_{j = p+1}^q \inner{u_j|f}_{L^2\p{\mu}} u_j(x)}\\
				& = \abs{\sum_{j = p+1}^q \p{1 + \lambda_j}^{-\frac{\alpha + d/2}{2}} \p{1 + \lambda_j}^{\frac{\alpha + d/2}{2}} \inner{u_j|f}_{L^2\p{\mu}} u_j(x)} \\
				& \leq \abs{\sum_{j = p+1}^q \p{1 + \lambda_j}^{\alpha + d/2} \inner{u_j|f}_{L^2\p{\mu}}^2}^{1/2} \abs{\sum_{j = p+1}^q \p{1 + \lambda_j}^{-\p{\alpha + d/2}} u_j^2(x)}^{1/2} \\
				& \leq \norm{f}_{H^{\alpha+d/2}\p{\XX}} \abs{\sum_{j > p} \p{1 + \lambda_j}^{-\p{\alpha + d/2}} u_j^2(x)}^{1/2}.
			\end{align*}
			We now use the representation
			\[
			\p{1 + \lambda_j}^{-\p{\alpha + d/2}} = \frac{1}{\Gamma\p{\alpha+d/2}} \int_0^\infty t^{\alpha+d/2-1} e^{-t} e^{-t\lambda_j} dt
			\]
			and Fubini's theorem to get
			\begin{align*}
			\sum_{j > p} \p{1 + \lambda_j}^{-\p{\alpha + d/2}} u_j^2(x) & \leq \frac{1}{\Gamma\p{\alpha+d/2}} \int_0^\infty t^{\alpha+d/2-1} e^{-t} \sum_{j > p} e^{-t\lambda_j} u_j^2(x) dt \\
			& \leq \frac{1}{\Gamma\p{\alpha+d/2}} \int_0^\infty t^{\alpha+d/2-1} e^{-t} \sum_{j > p} \ind{\lambda_j > c^2p^{2/d}} e^{-t\lambda_j} u_j^2(x) dt
			\end{align*}
			for some $c>0$, using Proposition \ref{proposition:growth_eigenpairs_dirichlet_spaces}. Now the proof of Lemma $3.19$ in \cite{coulhonHeatKernelGenerated2012} yields in particular
			\[
			\sum_{j > p} \ind{\b{\lambda_j > c^2p^{2/d}}} e^{-t\lambda_j} u_j^2(x) = \ind{\p{cp^{1/d}, \infty}}\p{\sqrt{L}}e^{-tL}\p{x,x} \lesssim \frac{e^{-t c^2 p^{2/d}/2}}{t^{d/2}},
			\]
			(where we have used the volume regularity from the set of assumptions \ref{assumptions:dirichlet_spaces}), and therefore
			\begin{align*}
			\sum_{j > p} \p{1 + \lambda_j}^{-\p{\alpha + d/2}} u_j^2(x) & \lesssim \int_0^\infty t^{\alpha - 1} e^{-t \p{1 + c^2 p^{2/d}/2}} dt \\
			& = \int_0^\infty \p{\frac{t}{1 + c^2 p^{2/d}/2}}^\alpha e^{-t} dt \\
			& \lesssim p^{-2\alpha/d},
			\end{align*}
			which implies
			\[
			\norm{S_q - S_p}_\infty \lesssim \norm{f}_{H^{\alpha+d/2}\p{\XX}} p^{-\alpha/d}.
			\]
			In particular $\p{S_p}_{p \geq 0}$ is a Cauchy sequence in the complete metric space $\p{\CC\p{\XX},\norm{\cdot}_\infty}$, and therefore converges to a limit $\tilde f \in \CC\p{\XX}$. Since we already know that $S_p \xrightarrow[p \to \infty]{L^2} f$ we must have $\tilde f = f$ $\mu-$almost everywhere. In order to show that $\norm{\tilde f}_\infty \lesssim \norm{f}_{H^{\alpha + d/2}\p{\XX}}$, notice that by the above reasoning with $p = 0$ we have, for any $q \geq 1$,
			\[
			\norm{S_q}_\infty \leq \norm{f}_{H^{\alpha + d/2}\p{\XX}} \sup_{x \in \XX} \abs{ \sum_{j \geq 1} \p{1 + \lambda_j}^{-(\alpha + d/2)} u_j^2(x)}^{1/2} \lesssim \norm{f}_{H^{\alpha + d/2}\p{\XX}},
			\]
			and by uniform convergence $S_q \xrightarrow[\CC\p{\XX}]{q \to \infty} \tilde f$ we deduce $\norm{\tilde f}_\infty \lesssim \norm{f}_{H^{\alpha+d/2}\p{\XX}}$, i.e. that we have a continuous embedding $H^{\alpha+d/2}\p{\XX} \hookrightarrow \CC\p{\XX}$.
			\item Notice that $\tilde u_j = \p{1 + \lambda_j}^{- \frac{\alpha + d/2}{2}} u_j, \quad j \geq 1$ form an orthonormal basis of $H^{\alpha+d/2}\p{\XX}$. Thus the Hilbert-Schmidt norm $\norm{\iota}_{HS}$ of $\iota$ can be bounded using Proposition \ref{proposition:growth_eigenpairs_dirichlet_spaces} :
			\[
			\norm{\iota}_{HS}^2 = \sum_{j \geq 1} \norm{\iota \tilde u_j}_2^2 = \sum_{j \geq 1} \norm{\p{1 + \lambda_j}^{-\frac{\alpha + d/2}{2}} u_j}_2^2 = \sum_{j \geq 1} \p{1 + \lambda_j}^{-(\alpha + d/2)} \lesssim \sum_{j \geq 1} j^{-(1 + 2\alpha/d)} < +\infty,
			\]
			proving that $\iota$ is Hilbert-Schmidt and therefore compact (see e.g. Chapter 6 in \cite{brezisFunctionalAnalysisSobolev2011a}.
		\end{itemize}
	\end{proof}
	\begin{remark}
		Only continuity of $H^{\alpha+d/2}\p{\XX} \hookrightarrow \CC\p{\XX}$ and compactness of $H^{\alpha+d/2}\p{\XX} \hookrightarrow L^2$ are needed in the proof of Theorem \ref{corollary:matern_compact_dirichlet_spaces}, but it is potentially possible to show using only the set of assumptions \ref{assumptions:dirichlet_spaces} that the embedding $H^{\alpha+d/2}\p{\XX} \hookrightarrow \CC\p{\XX}$ is itself compact for any $\alpha > 0$, thus strengthening the results of Proposition \ref{prop:embeddings} (this is known to be the case for instance for compact manifolds, see e.g. \cite{devitoReproducingKernelHilbert2021a} and the references therein).
	\end{remark}
	
	\begin{remark}\label{remark:interpolation_sobolev_spaces}
	The Sobolev spaces $H^s\p{\XX}, \quad s > 0$ satisfy the interpolation property
	\[
	\p{L^2,H^s\p{\XX}}_{\theta,2} \simeq H^{\theta s}\p{\XX}, \quad \theta \in \p{0,1}, \quad s > 0.
	\]
	Indeed, since by Proposition \ref{prop:characterisation_besov_spaces} we have $H^s\p{\XX} \simeq B_{22}^s$, this can be seen as a consequence of the reiteration Theorem (see e.g. Remark $1.3.6$ in \cite{lunardiInterpolationTheory2018}) coupled with the characterisation of the Besov spaces $B_{22}^s$ in terms of real interpolation spaces given in Proposition $6.2$ and Theorem $3.16$ in \cite{coulhonHeatKernelGenerated2012}.
	\end{remark}
	
	\subsection{Proofs of Section \ref{sec:gaussian_processes}}\label{appendix:proof_GPs}
	
	\begin{lemma}\label{lemma:generic_matern_bound_basis}
		Under Assumptions \ref{assumption:EVD} \& \ref{assumption:EMB} of Section \ref{subsection:general_gps}, the basis $\p{e_j}_{j \geq 1}$ satisfies
		\[
		\sup_{x \in \XX} \sum_{j = 1}^J e_j^2(x) \leq 2 c_1^{-\tau} \norm{\Hbb_{\mu_0}^\tau \hookrightarrow L^\infty\p{\mu_0}}^2 J^{\tau/p}, \quad J \geq 1.
		\]
	\end{lemma}
	\begin{proof}
		Using the proof of Lemma 6.5 in \cite{fischerSobolevNormLearning2020a} we obtain
		\[
		\sum_{j \geq 1} \frac{s_j}{s_j + \lambda} e_j^2(x) \leq \norm{\Hbb_{\mu_0}^\tau \hookrightarrow L^\infty\p{\mu_0}}^2 \lambda^{-\tau}, \quad \lambda > 0,
		\]
		where the constant $\norm{\Hbb_{\mu_0}^\tau \hookrightarrow L^\infty\p{\mu_0}}^2$ is finite using Assumption \ref{assumption:EMB}. Now using Assumption \ref{assumption:EVD} we have $s_j \geq c_1 j^{-1/p} \geq c_1 J^{-1/p}$ for any $j \leq J$, therefore
		\[
		\sum_{j \leq J} u_j^2(x) \leq \sum_{s_j \geq c_1 J^{-1/p}} e_j^2(x) \leq \sum_{j \geq 1} \frac{2s_j}{s_j + c_1 J^{-1/p}} e_j^2(x) \leq 2c_1^{-\tau}\norm{\Hbb_{\mu_0}^\tau \hookrightarrow L^\infty\p{\mu_0}}^2 J^{\tau/p},
		\]
		where the second inequality in the last display holds because $\mu \geq \bar \mu > 0$ implies $1 \leq \frac{2\mu}{\mu + \bar \mu}$.
	\end{proof}
	
	\begin{theorem}\label{theorem:cv_krr}
		In the setting of Section \ref{subsection:general_gps}, if $\lambda \asymp \frac{1}{n}$ and $\hat f \in \Hbb$ is the kernel ridge regression estimator
		\[
		\hat f \in \argmin_{f \in \Hbb} \lambda \norm{f}_\Hbb^2 + \frac{1}{n} \sum_{i=1}^n \p{y_i - f(x_i)}^2,
		\]
		then if Assumptions \ref{assumption:EMB} \ref{assumption:EVD} \ref{assumption:SRC} and \ref{assumption:MOM} hold with either $0 < p < \tau < \beta < \beta + p \leq 1$ or $0 < \beta \leq p < \tau < \beta + p \leq 1$, for any $\rho \geq 1$ and for a constant $K > 0$ independent of $n$ and $\rho$, with probability at least $1-4e^{-\rho}$,
		\[
		\norm{\hat f - f_0}_{L^2\p{\mu_0}} \leq K \rho^2 n^{-\beta}.
		\]
	\end{theorem}
	\begin{proof}
		This is a slight variation on Theorem 1 in \cite{fischerSobolevNormLearning2020a}, so we omit the details and only point out the key differences. We still have $\frac{\ln \lambda^{-1}}{n\lambda^\tau} \sim n^{\tau - 1}\ln n \to 0$ because $\tau < 1$ by assumption, and therefore following the proof of Theorem 1 in \cite{fischerSobolevNormLearning2020a} we find that, with probability at least $1-4e^{-\rho}$,
		\begin{align*}
			\norm{f_0 - \hat f}_{L^2\p{\mu_0}}^2 & \lesssim \lambda^\beta + \norm{f_{P,\lambda} - f_{D,\lambda}}_{L^2\p{\mu_0}}^2 \\
			& \lesssim \lambda^\beta + \frac{\rho^2}{n} \p{\lambda^{-p} + \lambda^{\beta - \tau} + \frac{\lambda^{-\p{\tau - \beta}_+}}{n \lambda^\tau}} \\
			& \lesssim \rho^2 n^{-\beta} \p{n^{\beta + p-1} + n^{\tau - 1} + n^{\p{\tau - \beta}_+ + \beta + \tau - 2}} \\
			& \lesssim \rho^2 n^{-\beta} \p{1 + n^{\beta + \tau + \p{\tau - \beta}_+ - 2}}.
		\end{align*}
		We now distinguish the cases :
		\begin{itemize}
			\item If $0 < p < \tau < \beta < \beta + p \leq 1$ then $\p{\tau - \beta}_+ = 0$ and therefore $\beta + \tau + \p{\tau - \beta}_+ - 2 = \beta + \tau - 2 < 0$, which yields $\norm{f_0 - f_{P,\lambda}}_{L^2\p{\mu_0}}^2 \lesssim \rho^2n^{-\beta}$.
			\item If $0 < \beta \leq p < \tau < \beta + p \leq 1$ then $\p{\tau - \beta}_+ = \tau - \beta$ and therefore $\beta + \tau + \p{\tau - \beta}_+ - 2 = 2\p{\tau - 1} < 0$, which also yields $\norm{f_0 - f_{P,\lambda}}_{L^2\p{\mu_0}}^2 \lesssim \rho^2n^{-\beta}$.
		\end{itemize}
	\end{proof}
	
	\begin{proof}{Theorem \ref{thm:general_thm_GPs}}\label{proof:general_thm_GPs}\\
		By Theorem \ref{theorem:cv_krr}, for any $\sigma > 0$ the kernel regression estimator $\hat f$ defined via
		\[
		\hat f := \argmin_{f \in \Hbb} \sigma^2 \norm{f}_\Hbb^2 + \sum_{i=1}^n \p{y_i - f(x_i)}^2
		\]
		satisfies $\norm{\hat f - f_0}_{L^2\p{\mu_0}} = \mathcal{O}_{P_0}\p{M_n \varepsilon_n}$. But $\hat f$ is the posterior mean of $f$ given $\p{X,Y}$ (see e.g. \cite{kanagawaGaussianProcessesReproducing2025} Corollary 6.7), therefore
		\[
		\int \norm{f-f_0}_{L^2\p{\mu_0}}^2 \Pi \c{df | X,Y} = \norm{\hat f - f_0}_{L^2\p{\mu_0}}^2 + \int \norm{f - \hat f}_{L^2\p{\mu_0}}^2 \Pi \c{df | X,Y}.
		\]
		We will conclude by bounding the second term in the last display. By Gaussian conditioning (see e.g. \cite{kanagawaGaussianProcessesReproducing2025} Theorem 6.1), the posterior distribution of $f$ given $\p{x_i,y_i}_{i=1}^n$ is again Gaussian with mean $\hat f$ and covariance kernel $k_n$ given by
		\[
		k_n(x,y) = k(x,y) - k_{Xx}^T A^{-1} k_{Xy}, \quad x,y \in \XX,
		\]
		where
		\[
		A := k_{XX} + \sigma^2 I_n \in \RR^{n \times n}, \quad k_{XX} := \c{k(x_i,x_j)}_{i,j=1}^n \in \RR^{n \times n}, \quad k_{Xx} := \c{k(x_1,x),\ldots,k(x_n,x)}^T \in \RR^n.
		\]
		Thus by Fubini's theorem we have
		\[
		\sigma_n^2 := \int \norm{f - \hat f}_{L^2\p{\mu_0}}^2 \Pi \c{df | X,Y} = \int \int \p{\hat f (x) - f(x)}^2 \mu_0(dx) \Pi \c{df|X,Y} = \int k_n(x,x) \mu_0(dx),
		\]
		and therefore by linearity, with
		\[
		B := \int k_{Xx} k_{Xx}^T \mu_0(dx) \in \RR^{n \times n}
		\]
		we have
		\[
		\sigma_n^2 = \int k(x,x)\mu_0(dx) - \int \Tr \c{A^{-1} k_{Xx} k_{Xx}^T} \mu_0(dx) = \int k(x,x)\mu_0(dx) - \Tr \c{A^{-1} B}.
		\]
		Now using Mercer's theorem (see e.g. \cite{kanagawaGaussianProcessesReproducing2025}) we can write
		\[
		k(x,y) = \sum_{j \geq 1} s_j e_j(x) e_j(y), \quad x,y \in \XX,
		\]
		with absolute and uniform convergence over $\XX^2$, and therefore the first term of the last display is given by
		\[
		\int k(x,x)\mu_0(dx) = \sum_{j \geq 1} s_j.
		\]
		Moreover, for any $i,j = 1,\ldots,n$, by orthonormality of the $u_j's$,
		\begin{align*}
			B_{ij} & = \int k(x,x_i)k(x,x_j)\mu_0(dx) \\
			& = \int \p{\sum_{l \geq 1} s_l u_l(x) e_l(x_i)} \p{\sum_{l \geq 1} s_l e_l(x) u_l(x_j)} \mu_0(dx) \\
			& = \sum_{l \geq 1} s_l^2 e_l(x_i)e_l(x_j) \\
			& = \sum_{l \geq 1} s_l^2 \p{e_{l,X} e_{l,X}^T}_{ij},
		\end{align*}
		where
		\[
		e_{l,X} := \c{e_l(x_1),\ldots,e_l(x_n)}^T.
		\]
		Therefore
		\[
		\int k_n(x,x)\mu_0(dx) = \sum_{l \geq 1} s_l \p{1 - s_l \Tr \c{A^{-1} e_{l,X} e_{l,X}^T}}.
		\]
		We will first show the following result : if $L \geq 1$ and $U_L \in \RR^{n \times L}$ is defined via $\p{U_L}_{il} := e_l(x_i)$, then on the event $E := \b{\frac{1}{n} U_L^T U_L \geq \kappa I_K}, \kappa > 0$, we have
		\[
		\int k_n(x,x)\mu_0(dx) \leq \frac{\sigma^2 L}{n \kappa} + \sum_{l > L} s_l.
		\]
		In order to prove this, notice that we can always assume without loss of generality that $s_l \in \c{0,1}$ : indeed, if not then consider $s_+ = \sup_{l \geq 1} s_l > 1$ and
		\[
		\sigma_n^2 := s_+ \tilde \sigma_n^2, \quad \tilde s_l := \frac{s_l}{s_+}, \quad \tilde A := s_+^{-1} A = \tilde \sigma^2 I_n + \tilde k_{XX}, \quad \tilde \sigma^2 = \frac{\sigma^2}{s_+}, \quad \tilde k = \frac{k}{s_+},
		\]
		so that
		\[
		\tilde \sigma_n^2 = \sum_{l \geq 1} \tilde s_l \p{1 - \tilde s_l \Tr \c{\tilde A^{-1} e_{l,X} e_{l,X}^T}},
		\]
		i.e. $\sigma_n^2 = \lambda_1 \int \tilde k_n(x,x)\mu_0(dx)$ where $\tilde k_n$ is the posterior kernel corresponding to the prior kernel $\tilde k$ and the Gaussian observation model with errors $\NN\p{0,\tilde \sigma^2}$. Thus, if the result is shown for the case $s_+ \leq 1$, we get $\int \tilde k_n(x,x) \mu_0(dx) \leq \frac{\tilde \sigma^2 L}{n \kappa} + \sum_{l > L} \tilde s_l$, which then implies $\sigma_n^2 = s_+ \tilde \sigma_n^2 \leq \frac{\sigma^2 L}{n \kappa} + \sum_{l > L} s_l$.
		
		For each $L$ we have $A = \sigma^2 I_n + k_{XX} \succeq \sigma^2 I_n + k_{XX}^{(L)} =: A_L$ where
		\[
		\p{k_{XX}^{(L)}}_{ij} = \sum_{l = 1}^L s_l e_l(x_i)e_l(x_j), \quad i,j=1,\ldots,n.
		\]
		This implies
		\begin{align*}
			\sigma_n^2 & = \sum_{l = 1} ^L s_l \p{1 - s_l \Tr \c{A^{-1} e_{l,X} e_{l,X}^T}} + \sum_{l > L} s_l \p{1 - s_l \Tr \c{A^{-1} e_{l,X} e_{l,X}^T}} \\
			& =  \sum_{l = 1} ^L s_l \p{1 - s_l \Tr \c{A_L^{-1} e_{l,X} e_{l,X}^T}} +  \sum_{l = 1} ^L s_l^2 \Tr \c{ \p{A_L^{-1} - A^{-1}} e_{l,X} e_{l,X}^T} \\
			& + \sum_{l > L} s_l \p{1 - s_l \Tr \c{A^{-1} e_{l,X} e_{l,X}^T}} \\
			& = \Tr \c{M_L - A_L^{-1} U_L M_L^2 U_L^T} + \Tr \c{ \p{A_L^{-1} - A^{-1}} U_L M_L^2 U_L^T } + \Tr \c{M_{>L} - A^{-1} U_{>L} M_{>L} U_{>L}^T} \\
			& = \underbrace{\Tr \c{M_L - A_L^{-1} U_L M_L^2 U_L^T}}_{=: (I)} + \underbrace{\Tr \c{ \p{A_L^{-1} - A^{-1}} U_L M_L^2 U_L^T - A^{-1} K_{>L}}}_{=: (II)} + \sum_{l > L} \mu_l,
		\end{align*}
		where we have defined
		\[
		M_L := \Diag\p{\mu_l : 1\leq l \leq L}, \quad M_{>L} = \Diag\p{s_l : l > L}, \quad \p{U_{>L}}_{il} = e_l(x_i), \quad l > L, i = 1,\ldots,n
		\]
		and
		\[
		\p{K_{>L}}_{ij} := \sum_{l > L} s_l e_l(x_i) e_l(x_j).
		\]
		We start by bounding $\p{I}$. We use the identity
		\[
		A_L^{-1} U_L M_L = U_L \p{\sigma^2 M_L^{-1} + U_L^T U_L}^{-1}
		\]
		(easily proven by multiplying each side by the relevant factors, since $A_L = \sigma^2 I_n + U_L M_L U_L^T$), to get (using cyclicity property of the trace), on the event $E$,
		\begin{align*}
			\p{I} & = \Tr \c{M_L - U_L \p{\sigma^2 M_L^{-1} + U_L^T U_L}^{-1}M_L U_L^T} \\
			& = \Tr \c{M_L - \p{\sigma^2 M_L^{-1} + U_L^T U_L}^{-1}M_L U_L^T U_L} \\
			& = \Tr \c{ \p{\sigma^2 M_L^{-1} + U_L^T U_L} \p{\sigma^2 M_L^{-1} + U_L^T U_L}^{-1} M_L - \p{\sigma^2 M_L^{-1} + U_L^T U_L}^{-1}M_L U_L^T U_L} \\
			& = \Tr \c{ \p{\sigma^2 M_L^{-1} + U_L^T U_L}^{-1} M_L \p{\sigma^2 M_L^{-1} + U_L^T U_L} - \p{\sigma^2 M_L^{-1} + U_L^T U_L}^{-1}M_L U_L^T U_L} \\
			& = \Tr \c{\p{\sigma^2 M_L^{-1} + U_L^T U_L}^{-1} \c{M_L \p{\sigma^2 M_L^{-1} + U_L^T U_L} - M_L U_L^T U_L }} \\
			& = \sigma^2 \Tr \c{\p{\sigma^2 M_L^{-1} + U_L^T U_L}^{-1}} \\
			& = \frac{\sigma^2}{n} \Tr \c{\p{\frac{\sigma^2}{n} M_L^{-1} + \frac{1}{n} U_L^T U_L}^{-1}} \\
			& \leq \frac{\sigma^2 L}{n} \lambda_{\max}\p{\p{\frac{\sigma^2}{n} M_L^{-1} + \frac{1}{n} U_L^T U_L}^{-1}} \\
			& = \frac{\sigma^2 L}{n} \lambda_{\min}\p{\frac{\sigma^2}{n} M_L^{-1} + \frac{1}{n} U_L^T U_L}^{-1} \\
			& \leq \frac{\sigma^2 L}{n} \lambda_{\min}\p{\frac{1}{n} U_L^T U_L}^{-1} \\
			& \leq \frac{\sigma^2 L}{n \kappa}.
		\end{align*}
		Finally we show that $\p{II}$ is non positive. We have
		\[
		A_L^{-1} - A^{-1} = A_L^{-1} - \p{A_L + K_{>L}}^{-1} = A^{-1} K_{>L} A_L^{-1},
		\]
		which can be checked by multiplying each side of this equality by the relevant factors. In particular, the matrix $Q := A^{-1} K_{>L} A_L^{-1}$ is symmetric. We thus have
		\begin{align*}
		\p{II} & = \Tr \c{ A^{-1} K_{>L} A_L^{-1} \p{ U_L M_L^2 U_L^T - A_L}} \\
		& = \Tr \c{ A^{-1} K_{>L} A_L^{-1} \p{ U_L M_L^2 U_L^T - \sigma^2 I_n - U_L M_L U_L^T}} \\
		& = - \sigma^2 \Tr \c{Q} - \Tr \c{ A^{-1} K_{>L} A_L^{-1} \p{U_L M_L U_L^T - U_L M_L^2 U_L^T}} \\
		& = - \sigma^2 \Tr \c{Q} - \Tr \c{ A^{-1} K_{>L} A_L^{-1} U_L M_L \p{I_L - M_L}U_L^T},
		\end{align*}
		and with $S := A_L^{-1/2} K_{>L} A_L^{-1/2}$ we have $I+S = A_L^{-1/2} A A_L^{-1/2}$ (because $A_L + K_{>L} = A$), which implies $\p{I+S}^{-1}S = A_L^{1/2} A^{-1} K_{>L} A_L^{-1/2}$ and therefore
		\[
		A_L^{-1/2}\p{I+S}^{-1}SA_L^{-1/2} = A^{-1} K_{>L} A_L^{-1}.
		\]
		In particular, the matrix $Q := A^{-1} K_{>L} A_L^{-1}$ is not only symmetric, but also positive semidefinite (because $\p{I+S}^{-1}S$ is and $A_L^{-1/2}$ is symmetric). Notice that this was not easily seen from the representation $Q = A_L^{-1} - A^{-1}$ as $A_L$ and $A$ do not necessarily commute. Thus $\Tr \c{Q} \geq 0$ and if $Q^{1/2}$ is a square root of $Q$ we have by cyclicity
		\[
		\p{II} = - \sigma^2 \Tr \c{Q} -\Tr \c{Q^{1/2} U_L M_L\p{I_L - M_L} U_L^TQ^{1/2}} \leq 0
		\]
		precisely because $M_L$ has eigenvalues $s_l \in \c{0,1}$, and therefore
		\[
		\sigma_n^2 \leq \frac{\sigma^2 L}{n \kappa} + \sum_{l > L} s_l.
		\]
		Now using Assumption \ref{assumption:EVD} we have
		\[
		\sum_{l > L} s_l \leq c_2 \sum_{l > L} l^{-1/p} \leq \frac{c_2 p}{p-1} L^{1-1/p},
		\]
		so that choosing $L \sim n^p$ leads to $\sigma_n^2 \ind{E} \lesssim n^{p-1}$. We now show that $\Pro_0 \p{E} \to 1$ with this choice of $L$. For this, using Assumptions \ref{assumption:EVD} \& \ref{assumption:EMB} and Lemma \ref{lemma:generic_matern_bound_basis} we get
		\[
		\forall L \geq 1, \quad C_L^2 := \frac{1}{L} \sup_{x \in \XX} \sum_{l=1}^L e_l^2(x) \leq 2\norm{\Hbb_{\mu_0}^\tau \hookrightarrow L^\infty\p{\mu_0}}^2 L^{\tau/p - 1}.
		\]
		But then using Theorem \ref{theorem:1} (with $\Sigma = I_L, p_- = p_+ = 1$) we have $\Pro_0 \p{E^c} \leq e^{-c_0n/C_L^2 L}$ for a universal constant $c_0$ and $\kappa = 1/2$, since $C_L^2 L \ln L \asymp L^{\tau/p} \ln L \asymp n^\tau \ln n \ll n$, since $\tau < 1$. We conclude using Markov' inequality :
		\begin{align*}
			\EE_0 \Pi \c{\norm{f-f_0}_{L^2\p{\mu_0}} > M_n \varepsilon_n | X,Y} & \leq \EE_0 \Pi \c{\norm{f-\hat f}_{L^2\p{\mu_0}} > M_n \varepsilon_n/2 | X,Y} + \Pro_0 \p{\norm{\hat f - f_0}_{L^2\p{\mu_0}} > M_n \varepsilon_n/2} \\
			& \leq \Pro_0\p{E^c} + \EE_0 \ind{E} \Pi \c{\norm{f-\hat f}_{L^2\p{\mu_0}} > M_n \varepsilon_n/2 | X,Y} + o(1) \\
			& \leq \frac{4}{M_n^2 \varepsilon_n^2} \EE_0 \c{\ind{E} \int \norm{f - f_0}_{L^2\p{\mu_0}}^2 \Pi \c{df | X,Y}} + o(1) \\
			& = \frac{4}{M_n^2 \varepsilon_n^2} \EE_0 \c{\sigma_n^2 \ind{E}} + o(1) \\
			& \lesssim \frac{n^{p-1}}{M_n^2 \varepsilon_n^2} + o(1) \\
			& = \frac{1}{M_n^2} + o(1) \\
			& = o(1).
		\end{align*}
	\end{proof}
	
	\begin{proof}{Corollary \ref{corollary:matern_domains}}\\
		Since $H^s\p{\XX} \hookrightarrow H^{s \wedge \alpha}\p{\XX}$ we can assume without loss of generality that $s \leq \alpha$. Following \cite{fischerSobolevNormLearning2020a} Section 4, Assumption \ref{assumption:EVD} holds with $1/p = 1 + 2\alpha/d \iff p = \frac{d}{2\alpha + d}$ and Assumption \ref{assumption:SRC} holds with $\beta := \frac{2s}{2\alpha + d}$. Moreover since $p_0$ is upper and lower bounded by positive constants, we have $L^2\p{\mu_0} \cong L^2\p{\mu}$ and therefore $\Hbb_{\mu_0}^\tau \cong H^{\tau\p{\alpha + d/2}}$, which by the Sobolev embedding theorem (see \cite{adamsSobolevSpaces2003} Chapter $7$) satisfies $\Hbb^{\tau\p{\alpha + d/2}} \hookrightarrow L^\infty\p{\mu_0}$ because $\tau\p{\alpha + d/2} > d/2 \iff \tau > \frac{d}{2\alpha + d} = p$. To choose $\tau$ we distinguish the cases : if $s > d/2 \iff \beta > p$ we take any $\tau \in \p{p,\beta}$ while if $s \leq d/2 \iff \beta \leq p$ we take any $\tau \in \p{p,\beta + p}$, so that Assumption \ref{assumption:EMB} is satisfied and therefore the result follows by an application of Theorem \ref{thm:general_thm_GPs}.
	\end{proof}
	
	\begin{proof}{Corollary \ref{corollary:matern_compact_dirichlet_spaces}}\\
		As for the proof of Corollary \ref{corollary:matern_domains}, we will show that Assumption \ref{assumption:EVD} holds with $1/p = 1+ 2 \alpha/d$, Assumption \ref{assumption:SRC} holds with $\beta = \frac{2s}{2\alpha + d}$, Assumption \ref{assumption:EMB} holds for any $\tau \in \p{p,\beta}$ when $s>d/2$ and any $\tau \in \p{p,\beta+p}$ when $s \leq d/2$, and conclude by applying Theorem \ref{thm:general_thm_GPs}. Again, we can assume without loss of generality that $s \leq \alpha$. First, by assumption we have $f_0 \in H^s\p{\XX} = H^{\beta\p{\alpha+d/2}}\p{\XX}$, which by Remark \ref{remark:interpolation_sobolev_spaces} is norm equivalent to $\p{L^2\p{\mu},H^{\alpha+d/2}\p{\XX}}_{\beta,2} \cong \p{L^2\p{\mu_0},\Hbb}_{\beta,2} = \Hbb_{\mu_0}^\beta$ (because $\Hbb \simeq H^{\alpha+d/2}\p{\XX}$ and the density $p_0$ is upper and lower bounded by positive constants). Therefore, Assumption \ref{assumption:SRC} holds.
		
		We now show that $\mu_j \asymp j^{-(1 + 2\alpha/d)}$, i.e. that Assumption \ref{assumption:EVD} holds with $p = \frac{d}{2\alpha+d}$. For this, notice that as in Section 4 of \cite{fischerSobolevNormLearning2020a} the eigenvalue $\mu_j$ is the square of the $j-$th approximation number $a_j\p{I_{\mu_0}}$ of the canonical inection $I_{\mu_0} : \Hbb \mapsto L^2\p{\mu_0}$, i.e.
		\[
		\mu_j = a_j^2\p{I_{\mu_0}}, \quad a_j\p{I_{\mu_0}} := \inf_{\substack{T \in \LL\p{\Hbb, L^2\p{\mu_0}} \\ \rk \p{T} < j}} \norm{I_{\mu_0} - T}.
		\]
		By assumption $\Hbb$ and $H^{\alpha+d/2}\p{\XX}$ are equal as sets and norm equivalent, so the identity map $\iota : \Hbb \mapsto H^{\alpha+d/2}\p{\XX}$ and its inverse $\iota^{-1} : H^{\alpha + d/2}\p{\XX} \to \Hbb$ are bounded as linear maps between Hilbert spaces. Thus for any linear operator $T$ of rank strictly less than $j$ we have
		\begin{align*}
			\norm{I_{\mu_0} - T}_{\LL\p{\Hbb,L^2\p{\mu_0}}} & = \sup_{\norm{f}_\Hbb \leq 1} \norm{f - Tf}_{L^2\p{\mu_0}} \\
			& \asymp \sup_{\norm{f}_{H^{\alpha + d/2}\p{\XX}} \leq 1} \norm{f - Tf}_{L^2\p{\mu_0}},
		\end{align*}
		and using the fact that $\frac{d\mu_0}{d\mu}$ is upper and lower bounded by positive constants, this implies
		\[
		\norm{I_{\mu_0} - T}_{\LL\p{\Hbb,L^2\p{\mu_0}}} \asymp \sup_{\norm{f}_{H^{\alpha + d/2}\p{\XX}} \leq 1} \norm{f - Tf}_{L^2\p{\mu}} = \norm{\tilde I_\mu-T}_{\LL\p{H^{\alpha + d/2}\p{\XX},L^2\p{\mu}}},
		\]
		where $\tilde I_\mu : H^{\alpha+d/2}\p{\XX} \mapsto L^2\p{\mu}$ is the canonical injection. Therefore
		\[
		a_j\p{I_{\mu_0}} = \inf_{\substack{T \in \LL\p{\Hbb, L^2\p{\mu_0}} \\ \rk \p{T} < j}} \norm{I_{\mu_0} - T} \asymp \inf_{\substack{T \in \LL\p{H^{\alpha+d/2}\p{\XX}, L^2\p{\mu}} \\ \rk \p{T} < j}} \norm{\tilde I_\mu - T} = a_j\p{\tilde I_\mu},
		\]
		and it remains to show that $a_j\p{\tilde I_{\mu_0}} \asymp j^{-(1/2+\alpha/d)}$. But the reproducing kernel $\tilde k$ of $H^{\alpha + d/2}\p{\XX}$ is by definition given by
		\[
		\tilde k (x,y) = \sum_{j \geq 1} \p{1 + \lambda_j}^{-(\alpha + d/2)} u_j(x)u_j(y),
		\]
		so that the eigenvalues $\tilde \mu_j$ of the integral operator
		\[
		\tilde T_\mu : \p{\begin{array}{ccc}
				L^2\p{\mu} & \to & L^2\p{\mu} \\
				f & \mapsto & \int \tilde k(\cdot,y) f(y) \mu(dy)
		\end{array}}
		\]
		are actually given by $\tilde \mu_j = \p{1 + \lambda_j}^{-(\alpha + d/2)}$. Using Proposition \ref{proposition:growth_eigenpairs_dirichlet_spaces} this yields $\tilde \mu_j \asymp j^{-(1 + 2\alpha/d)}$, and using again the results of Section 4 in \cite{fischerSobolevNormLearning2020a} (but this time for $H^{\alpha+d/2}, \tilde k$ and $\mu$ instead of $\Hbb, k$ and $\mu_0$) this implies
		\[
		a_j^2\p{I_{\mu_0}} \asymp a_j^2\p{\tilde I_\mu} = \tilde \mu_j \asymp j^{-(1+2\alpha/d)},
		\]
		so that Assumption \ref{assumption:EVD} holds with $p = \frac{d}{2\alpha + d}$.
		
		Finally, since $p_0$ is upper and lower bounded $L^2\p{\mu_0} \cong L^2\p{\mu}$ which in turns implies, using $\Hbb \simeq H^{\alpha + d/2}\p{\XX}$,
		\[
		\Hbb_{\mu_0}^\tau = \p{L^2\p{\mu_0}, \Hbb}_{\tau,2} \cong \p{L^2\p{\mu}, H^{\alpha+d/2}\p{\XX}}_{\tau,2} \cong H^{\tau\p{\alpha+d/2}}\p{\XX},
		\]
		which is continuously embedded in $\CC\p{\XX}$ using Proposition \ref{prop:embeddings}, since $p = \frac{d}{2\alpha + d}$ we have $\tau \p{\alpha + d/2} > d/2$ because $\tau > p$ by assumption, proving that Assumption \ref{assumption:EMB} holds as well. The final result then follows by an application of Theorem \ref{thm:general_thm_GPs}.
	\end{proof}
	
	\subsection{Proofs of Section \ref{section:random_series_priors}}\label{appendix:proof_random_series}
	
	\begin{proof}{Theorem \ref{theorem:3}}\label{appendix:proof_thm3}\\
		The proof relies on standard arguments in order to check the assumptions of Theorem \ref{theorem:1}. We first prove a $\norm{\cdot}_n-$contraction result : first, by Lemma 2.7 in \cite{ghosalFundamentalsNonparametricBayesian2017}  and with $P_f\p{\cdot|x}$ the distribution $\otimes_{i=1}^n\NN\p{f\p{x_i},\sigma^2}$ we have
		\[
		KL\p{P_{f_0},P_f} = \frac{n}{2\sigma^2} \norm{f-f_0}_n^2 \leq n\sigma^{-2}\norm{f-f_0}_\infty^2, \quad V_{2,0}\p{P_{f_0},P_f} = \frac{n}{\sigma^2} \norm{f-f_0}_n^2 \leq n\sigma^{-2}\norm{f-f_0}_\infty^2,
		\]
		so that $\Pi \c{KL\p{P_{f_0},P_f} \vee V_{2,0} \p{P_{f_0},P_f} \leq n\varepsilon^2} \geq \Pi \c{\norm{f-f_0}_\infty \leq \sigma \varepsilon}$. Using Assumption \ref{assumption:4} for any $J \geq 1$ we can find $f_J \in V_J$ such that $\norm{f_0 - f_J}_\infty \leq KJ^{-s/d}$. Given Assumption \ref{assumption:3} on $\Psi$ we then easily get the prior mass lower bound : for any $J_n \geq 1$, if $K J_n^{-s/d} \leq \sigma \varepsilon/2$ we have with $f_n := f_{J_n}$
		\begin{align*}
			\Pi \c{KL\p{P_{f_0},P_f} \vee V_{2,0} \p{P_{f_0},P_f} \leq n\varepsilon^2} & \geq \Pi \c{\norm{f-f_0}_\infty \leq \sigma \varepsilon} \\
			& \geq \Pi \c{J = J_n} \Pi \c{\norm{f-f_0}_\infty \leq \sigma \varepsilon|J = J_n} \\
			& \geq \Pi \c{J = J_n} \Pi \c{\norm{f-f_n}_\infty \leq \sigma \varepsilon/2|J = J_n} \\
			& \geq \Pi \c{J = J_n} \Pi \c{\forall 1 \leq j \leq J_n, \quad \abs{Z_j - \inner{u_j|f_n}_2} \leq \sigma \varepsilon/2J_n|J = J_n} \\
			& = \Pi \c{J = J_n} \prod_{j=1}^{J_n} \int_{\inner{u_j|f_n}_2 - \varepsilon/2J_n}^{\inner{u_j|f_n}_2 + \varepsilon/2J_n} \Psi(x)dx \\
			& \geq \Pi \c{J = J_n} \prod_{j=1}^{J_n} \inf_{\c{\inner{u_j|f_n}_2 - \varepsilon/2J_n,\inner{u_j|f_n}_2 + \varepsilon/2J_n}} \Psi,
		\end{align*}
		and since
		\[
		\sup_{n,j} \abs{\inner{u_j|f_n}_2} \leq \norm{f_n}_2 \leq \norm{f_0}_2 + \sup_n KJ_n^{-s/d} \leq \norm{f_0}_2 + K,
		\]
		since $\Psi$ is positive and continuous we get for some $c>0$
		\[
		\Pi \c{KL\p{P_{f_0},P_f} \vee V_{2,0} \p{P_{f_0},P_f} \leq n\varepsilon^2} \geq \Pi \c{J = J_n} \exp \p{-cJ_n \ln \p{\varepsilon/J_n}}.
		\]
		If the prior on $J$ is deterministic we take $J_n$ as the precise value of the truncation, and if it is random we take $J_n = \floor{\p{n/\ln n}^{\frac{d}{2\beta+d}}} \asymp n\varepsilon_n^2/\ln n$. In both cases (and using Assumption \ref{prior:J} in the random $J$ case) we see that for some possibly large $c>0$,
		\[
		\Pi \c{KL\p{P_{f_0},P_f} \vee V_{2,0} \p{P_{f_0},P_f} \leq nc^2\varepsilon_n^2} \geq \exp \p{-nc^2 \varepsilon_n^2}.
		\]
		By now defining the set
		\[
		\FF_n = \b{\sum_{1 \leq j \leq u J_n} a_j u_j : \abs{a_j} \leq \p{v n\varepsilon_n^2}^{1/c_3}}, u \geq 1, v > 0,
		\]
		we get that
		\[
		\Pi \c{\FF_n^c} \leq \Pi \c{J > c_1 J_n} + \sum_{j \leq c_1 J_n} \Pi \c{\abs{Z_j} > u}.
		\]
		Given our prior on $J$ the first term of the last display is either $0$ or less than a multiple of $e^{-b u n \varepsilon_n^2} $ for some $b>0$, while using Assumption \ref{assumption:3} the second term can be bounded by a multiple of $e^{-c_2 vn \varepsilon_n^2}$. For any fixed values of $u,v$ we have $\norm{u_j}_\infty \leq C_j \sqrt{j} = o\p{j^{s/d + 1/2}}$ and therefore the $\norm{\cdot}_\infty$-metric entropy of $\FF_n$ can be seen to be bounded by a multiple of $J_n \ln n$, which is always upper bounded by a multiple of $n\varepsilon_n^2$ in either the random or deterministic truncation cases. Moreover, taking $u,v$ large enough implies $\Pi \c{\FF_n^c} \leq e^{-5c^2 n \varepsilon_n^2}$. Finally, as recalled in Section \ref{section:intro} the $\norm{\cdot}_n$ semi norm satisfies the testing condition \ref{eqn:testing_condition} in model \ref{model:nonparametric_regression}. Applying the testing approach (see e.g. \cite{ghosalConvergenceRatesPosterior2007} or Theorem 1.3 in \cite{castilloBayesianNonparametricStatistics2024}) then gives the existence of $M>0$ such that
		\[
		\EE_0 \Pi \c{\norm{f-f_0}_n > M \varepsilon_n | \Xbb^n} \to 0.
		\]
		Moreover an application Lemma 1.2 in \cite{castilloBayesianNonparametricStatistics2024} together with our previous estimates implies that
		\[
		\EE_0 \Pi \c{\Ecal_{uJ_n}\p{f} = 0 |\Xbb^n} \geq \EE_0\Pi \c{\FF_n|\Xbb_n} \to 1,
		\]
		while the estimate $\Ecal_{J_n}\p{f_0} \leq M_1 \varepsilon_n$ follows by Assumption \ref{assumption:4}.
		Since in both cases
		\[
		n \geq c_0 C_{J_n}^2 J_n \ln J_n \norm{\Sigma}^{-1} \p{p_+/p_-}^2
		\]
		for large $n$ (because $C_j = o \p{j^{s/d}}$ by assumption), we get the desired result by applying Theorem \ref{theorem:2}.
	\end{proof}
	
	\begin{proof}{Theorem \ref{theorem:L2_norm_castillo}}\\
		The proof is similar to the one of Theorem \ref{theorem:3} : as a consequence of the results in
		in \cite{castilloThomasBayesWalk2014} (more precisely the prior mass lower bound obtained in Section 5 \textit{"General conditions for posterior rates"}), we have for some $M'>0$,
		\[
		\EE_0 \Pi \c{\norm{f-f_0}_n > M' \varepsilon_n | \Xbb^n} \to 0,
		\]
		as well as
		\begin{equation}\label{eqn:prior_mass}
			\Pi \c{\norm{f-f_0}_\infty \leq M'\varepsilon_n} \geq \exp \p{-n M'^2\varepsilon_n^2},
		\end{equation}
		On the other hand, with $P_f\p{\cdot|x}$ the distribution $\otimes_{i=1}^n\NN\p{f\p{x_i},\sigma^2}$ we have
		\[
		KL\p{P_{f_0},P_f} = \frac{n}{2\sigma^2} \norm{f-f_0}_n^2 \leq n\sigma^{-2}\norm{f-f_0}_\infty^2, \quad V_{2,0}\p{P_{f_0},P_f} = \frac{n}{\sigma^2} \norm{f-f_0}_n^2 \leq n\sigma^{-2}\norm{f-f_0}_\infty^2,
		\]
		so that $\Pi \c{KL\p{P_{f_0},P_f} \vee V_{2,0} \p{P_{f_0},P_f} \leq nM'^2\varepsilon_n^2} \geq \exp \p{-nM'^2\varepsilon_n^2}$, perhaps for a larger $M'>0$.
		
		We let $J_n \sim \varepsilon_n^{-d/s}$ so that points $1)$ \& $2)$ of Theorem \ref{theorem:2} are satisfied. It only remains show point $3)$ :
		\[
		\EE_0 \Pi \c{\Ecal_{J_n}\p{f} \geq \varepsilon_n | \Xbb^n} \to 0.
		\]
		We do so by applying Lemma 1.2 from \cite{castilloBayesianNonparametricStatistics2024}. Recall that \textit{a priori}, $f|t$ has the distribution of the series
		\[
		\sum_{j \geq 1} e^{-t\lambda_j/2} Z_ju_j, \quad Z_j \iid \NN\p{0,1},
		\]
		and that $t \sim \pi_1$. Therefore, for any $t_n$ and  $\varepsilon>0$,
		\begin{align*}
			\Pi \c{\Ecal_{J_n}\p{f} > \varepsilon} & = \int_0^\infty \Pi \c{\Ecal_{J_n}\p{f} > \varepsilon|t} \pi_1(t)dt \\
			& \leq \pi_1\p{t \leq t_n} + \sup_{t > t_n} \Pi \c{\Ecal_{J_n}\p{f} > \varepsilon|t}.
		\end{align*}
		Moreover, writing $f|t$ as the random series $\sum_{j\geq 1} e^{-t\lambda_j/2} Z_j u_j$ where $Z_j \iid \NN\p{0,1}$, for any $t>t_n$, thanks to the exponential factor we can easily derive an upper bound on $\Ecal_{J_n}\p{f}$ :
		\begin{align*}
			\Ecal_{J_n}\p{f} & \leq \norm{\sum_{j > J_n} e^{-t\lambda_j/2} Z_j u_j}_\infty \\
			& \leq \sum_{j > J_n} e^{-t\lambda_j/2} \abs{Z_j} \norm{u_j}_\infty.
		\end{align*}
		Using Proposition \ref{proposition:growth_eigenpairs_dirichlet_spaces} which as we remind the reader also imply the bound $\norm{u_j}_\infty \leq C \sqrt{j}$, we get
		\[
		\Ecal_{J_n}\p{f} \leq C\sum_{j > J_n} e^{-t\lambda_j/2} \sqrt{j} \abs{Z_j}.
		\]
		Moreover, using again Proposition \ref{proposition:growth_eigenpairs_dirichlet_spaces} above we have $\lambda_j \asymp j^{2/d}$, and therefore for a fixed $c>0$ we have
		\begin{align*}
			\Ecal_{J_n}\p{f} & \leq C\sum_{j > J_n} e^{-2ctj^{2/d}} \sqrt{j} \abs{Z_j} \\
			& \leq C e^{-ctJ_n^{2/d}} \sum_{j > J_n} e^{-ctj^{2/d}} \sqrt{j} \abs{Z_j},
		\end{align*}
		which implies for any $u > 0$ and $t > t_n$,
		\begin{align*}
			\Pi \c{\Ecal_{J_n}\p{f} > \varepsilon | t} & \leq \Pi \c{C e^{-ctJ_n^{2/d}} \sum_{j > J_n} e^{-ctj^{2/d}} \sqrt{j} \abs{Z_j} > \varepsilon | t} \\
			& = \Pi \c{ \sum_{j > J_n} e^{-ctj^{2/d}} \sqrt{j} \abs{Z_j} > C^{-1} e^{ctJ_n^{2/d}}  \varepsilon } \\
			& \leq \Pi \c{ \sum_{j > J_n} e^{-ct_n j^{2/d}} \sqrt{j} \abs{Z_j} > C^{-1} e^{ct_nJ_n^{2/d}}  \varepsilon } \\
			& \leq \exp \p{-e^{ct_nJ_n^{2/d}}u \varepsilon/C} \Pi \c{ \exp \p{u \sum_{j > J_n} e^{-ct_n j^{2/d}} \sqrt{j} \abs{Z_j} }}.
		\end{align*}
		By independence and a simple monotone convergence argument we find
		\begin{align*}
			\Pi \c{\Ecal_{J_n}\p{f} > \varepsilon | t} & \leq \exp \p{-e^{ct_nJ_n^{2/d} }u \varepsilon/C} \Pi_{j > J_n} \Pi \c{ \exp \p{u e^{-ct_n j^{2/d}} \sqrt{j} \abs{Z} }} \\
			& \leq \exp \p{-e^{ct_nJ_n^{2/d} }u \varepsilon/C} \Pi_{j > J_n} \exp \p{ u^2 e^{-2ct_nj^{2/d}}j/2} \\
			& = \exp \p{-e^{ct_nJ_n^{2/d} }u \varepsilon/C} \exp \p{ u^2 \sum_{j > J_n} e^{-2ct_nj^{2/d}}j/2},
		\end{align*}
		and by a series integral comparison,
		\begin{align*}
			\sum_{j > J_n} je^{-2ct_nj^{2/d}} & = t_n^{-d/2}\sum_{j > J_n} \p{t_n j^{2/d}}^{d/2} e^{-2ct_nj^{2/d}} \\
			& \leq t_n^{-d/2} \underbrace{\p{\sup_{x > 0} x^{d/2} e^{-x}}}_{=:c_d} \sum_{j > J_n} e^{-ct_nj^{2/d}} \\
			& \leq c_d t_n^{-d/2} \int_{J_n}^\infty e^{-ct_ns^{2/d}} ds \\
			& = c_d' t_n^{-d} \int_{t_n J_n^{2/d}}^\infty e^{-cx} x^{d/2-1} dx \\
			& \leq c_d'' t_n^{-d} \int_{t_n J_n^{2/d}} e^{-cx/2}dx \\
			& = c_d'' t_n^{-d} e^{-ct_nJ_n^{2/d}/2}.
		\end{align*}
		Therefore, for some $c_d>0$
		\[
		\sup_{t>t_n} \Pi \c{\Ecal_{J_n}\p{f} > \varepsilon | t} \leq \exp \p{-e^{ct_nJ_n^{2/d} }u \varepsilon/C} \exp \p{c_d u^2 t_n^{-d} e^{-ct_nJ_n^{2/d}/2}}
		\]
		which, taking $u = \p{e^{ct_nJ_n^{2/d} }\varepsilon/C}/\p{2c_d u^2 t_n^{-d} e^{-ct_nJ_n^{2/d}/2}}$ implies
		\begin{align*}
			\sup_{t > t_n} \Pi \c{\Ecal_{J_n}\p{f} > \varepsilon | t} & \leq \exp \p{-\p{e^{ct_nJ_n^{2/d} }\varepsilon/C}^2/\p{4c_d t_n^{-d} e^{-ct_nJ_n^{2/d}/2}}} \\
			& = \exp \p{-t_n^d e^{5ct_nJ_n^{2/d}/2 }\varepsilon^2 / 4c_dC^2}.
		\end{align*}
		On the other hand, given Assumption \ref{eqn:prior_t_castillo} on $\pi_1$ we have, as $t_n \to 0$,
		\[
		\pi_1\p{t \leq t_n} = \int_0^{t_n} \pi_1(t)dt \leq t_n \sup_{0 < t \leq t_n} \pi_1(t) \lesssim e^{-t_n^{-d/2} \ln^{1+d/2}\p{1/t_n}/2},
		\]
		which implies
		\begin{equation}\label{eqn:prior_mass_tails}
			\Pi \c{\Ecal_J\p{f}>\varepsilon} \lesssim \exp \p{-t_n^{-d/2} \ln^{1+d/2}\p{1/t_n}/2} + \exp \p{-t_n^d e^{5ct_nJ_n^{2/d}/2 }\varepsilon^2 / 4c_dC^2}.
		\end{equation}

		Thus our estimates \ref{eqn:prior_mass_tails} \& \ref{eqn:prior_mass} show that for any $K>0$, with $t_n$ a small enough multiple of $\p{n \varepsilon_n^2 / \ln^{1+d/2}n}^{-2/d}$ and $J_n$ a large enough multiple of $\p{t_n^{-1} \ln n}^{d/2} \sim \p{\ln n}^{-4/d^2} n \varepsilon_n^2$ we have
		\[
		\Pi \c{\Ecal_{J_n} > M\varepsilon_n} \leq e^{-K n \varepsilon_n^2},
		\]
		and therefore Lemma 1.2 from \cite{castilloBayesianNonparametricStatistics2024} implies that for $K > 0$ large enough
		\[
		\EE_0 \Pi \c{\Ecal_{J_n} > M\varepsilon_n} \to 0,
		\]
		with $J_n$ a constant multiple of $n^{\frac{d}{2s+d}} \p{\ln n}^{\frac{2s}{2s+d} - \frac{4}{d^2}}$. In particular since $n \geq c_0 C_{J_n}^2 J_n \ln J_n \norm{\Sigma}^{-1} \p{p_+/p_-}^2$ for large $n$ we can apply Theorem \ref{theorem:2}, which concludes the proof.
	\end{proof}
	
	\begin{proof}{Theorem \ref{theorem:rggs_d_N}}\\
		By Theorem \ref{theorem:rggs_d_n} there exists $M>0$ such that
		\[
		\EE_0 \Pi \c{\norm{f-f_0}_n > M\varepsilon_n | \Xbb^n} \to 0,
		\]
		where $\varepsilon_n$ is given in \ref{theorem:rggs_d_N}. Moreover, it is shown in the proof of Theorems $3.1 \& 3.4$ in \cite{rosaNonparametricRegressionRandom2024a} that for any $H_1>0$ there exists $M_1 > 0$ such that
		\[
		\Pro_0 \p{\Pi \c{\norm{f-f_0}_{L^\infty\p{\nu}} < M_1 \varepsilon_n} < e^{-nM_1^2 \varepsilon_n^2}} \leq n^{-H_1},
		\]
		the precise value of $H_1>0$ not really being important here as long as it is positive. For the prior with randomised $J$ considered here, for any $J_n > j_0$ we have
		\[
		\Pi \c{J > J_n} \leq \sum_{j > J_n} e^{-a_2 jL_j} \lesssim e^{-a_2 J_n \ln J_n},
		\]
		so that choosing $J_n$ a large enough multiple of $n\varepsilon_n^2/\ln n$ yields
		\[
		\EE_0 \Pi \c{J \leq J_n | \Xbb^n} \to 1,
		\]
		by a similar remaining mass type argument as the one used in the proof of Theorem \ref{theorem:2}, only on a event of probability at least $1 - n^{-H_1}$. For the prior with fixed $J$ we trivially have $\Pi \c{J \leq J_n | \Xbb^n} = 1$ $\Pro_0-$almost surely because $\Pi \c{J = J_n} = 1$. Now, for any $1 \leq J \leq J_n$, on the associated graph with bandwidth $h = h_J$ we have, for any $H_2>0$, for some $M_2>0$,
		\[
		\Pro_0 \p{\Ecal_{J_n}\p{f_0} > M_2 \varepsilon_n} \leq n^{-H_2},
		\]
		where $\Ecal_{J_n}\p{f_0} = \inf_{f \in V_{J_n}} \norm{f-f_0}_{L^\infty\p{\nu}}, \quad V_{J_n} = \Span \b{u_j : 1\leq j \leq J_n}$. We remind the reader that even though the eigenpairs $\p{\lambda_j,u_j}_{j=1}^N$ depend on $h$ (and therefore the function $f_{0n} : \XX_N \to \RR$ depends on $h$ as well), we keep the dependence implicit for ease of notation. In the case of a prior with fixed $J$ (and therefore fixed $h$) this causes no trouble, but in the case of a randomised $J$ we need to show a result uniform in $h \in \HH_n$ in order for the proof to remain valid, which we do by applying a suitable union bound. Specifically, we find that $\Ecal_{J_n}\p{f_0} \leq M_2 \varepsilon_n$ simultaneously for any $h \in \HH_{J_n}$ with a probability at least $1 - J_n n^{-H_2}$, which is $1 - o(1)$ for large $H_2$. Using the triangle inequality, with probability at least $1-o(1)$, for any $h \in \HH_n := \b{h_J : 1 \leq J \leq J_n}$ there exists $f_{0n} \in V_{J_n}$ such that $\norm{f_0 - f_{0n}}_{L^\infty\p{\nu}} \leq M_2 \varepsilon_n$, which in turn implies $\norm{f-f_{0n}}_n \leq M_2 \varepsilon_n + \norm{f-f_0}_n$ for every $f$.
		
		As a consequence, by the triangle inequality
		\[
		\EE_0 \Pi \c{\norm{f-f_{0n}}_n > \p{M_2 + M} \varepsilon_n | \Xbb^n} \leq \EE_0 \Pi \c{\norm{f-f_0}_n > M \varepsilon_n | \Xbb^n} \to 0,
		\]
		and therefore, for $M_3,M_4>0$,
		\begin{align*}
			\EE_0 \Pi \c{\norm{f-f_0}_N > M_3 \varepsilon_n | \Xbb^n} & \leq \EE_0 \Pi \c{\norm{f-f_{0n}}_N > \p{M_3-M_2} \varepsilon_n | \Xbb^n} \\
			& \leq  \EE_0 \ind{\forall h \in \HH_n, \sup_{f \in V_{J_n}} \frac{\norm{f}_N}{\norm{f}_n} \leq M_4} \Pi \c{\norm{f-f_{0n}}_N > \p{M_3-M_2} \varepsilon_n | \Xbb^n} \\
			& + \Pro_0 \c{ \exists h \in \HH_n, \quad \sup_{f \in V_{J_n}} \frac{\norm{f}_N}{\norm{f}_n} > M_4 } \\
			& \leq  \EE_0 \Pi \c{\norm{f-f_{0n}}_n > \frac{\p{M_3-M_2}}{M_4} \varepsilon_n | \Xbb^n} + J_n \max_{h \in \HH_n}\Pro_0 \c{\sup_{f \in V_{J_n}} \frac{\norm{f}_N}{\norm{f}_n} > M_4 }.
		\end{align*}
		We will now find $M_4$ to make the second term of the last display vanish and we will choose $M_3 > M_2 + M_4\p{M_2 + M}$ to get
		\[
		\EE_0 \Pi \c{\norm{f-f_0}_N > M_3 \varepsilon_n | \Xbb^n} \leq \EE_0 \Pi \c{\norm{f-f_0}_n > M \varepsilon_n | \Xbb^n} + o(1) = o(1).
		\]
		To do so we adopt a symmetrisation argument : since the $X_i's$ are \iidt, they are exchangeable and therefore for any $\sigma \in \Scal_N$ (permutation of $\b{1,\ldots,N}$) we have
		\[
		\Pro_0 \c{\sup_{f \in V_{J_n}} \frac{\norm{f}_N}{\norm{f}_n} > M_4} = \Pro_0 \c{\sup_{f \in V_{J_n}} \frac{\norm{f}_N}{\norm{f}_{n,\sigma}} > M_4},
		\]
		where
		\[
		\norm{f-f_0}_{n,\sigma}^2 := \frac{1}{n} \sum_{i=1}^n \p{f(x_{\sigma(i)}) - f_0(x_{\sigma(i)})}^2.
		\]
		Here the specific choice of $h \in \HH_n$ is implicit. As a consequence
		\[
		\frac{1}{N!} \sum_{\sigma \in \Scal_N} \Pro_0 \c{\sup_{f \in V_{J_n}} \frac{\norm{f}_N}{\norm{f}_{n,\sigma}} > M_4} = \Pro_0 \c{\sup_{f \in V_{J_n}} \frac{\norm{f}_N}{\norm{f}_n} > M_4},
		\]
		i.e. as $X_i \iid \mu_0$ and $\sigma$ is a uniformly sampled permutation of $\b{1,\ldots,N}$ we have
		\[
		\Pro_0 \c{\sup_{f \in V_{J_n}} \frac{\norm{f}_N}{\norm{f}_n} > M_4} =\Pro_{0,\sigma} \c{\sup_{f \in V_{J_n}} \frac{\norm{f}_N}{\norm{f}_{n,\sigma}} > M_4},
		\]
		where $\Pro_{0,\sigma}$ is the induced probability distribution on $\p{X_1,\ldots,X_N,\sigma}$.
		
		For every $f = \sum_{j=1}^{J_n} a_j u_j \in V_{J_n}$ we have
		\[
		\norm{f}_N^2 = \frac{1}{N} \sum_{i=1}^N f^2(x_i) = a^T \Sigma_n a, \quad \p{\Sigma_n}_{jk} := \frac{1}{N} \sum_{i=1}^N u_j(x_i) u_k(x_i),
		\]
		and
		\[
		\norm{f}_{n,\sigma}^2 = \frac{1}{n} \sum_{i=1}^n f^2(x_{\sigma(i)}) = a^T \Sigma_{n,\sigma} a, \quad \p{\Sigma_{n,\sigma}}_{jk} := \frac{1}{n} \sum_{i=1}^n u_j(x_{\sigma(i)}) u_k(x_{\sigma(i)}).
		\]
		We will conclude by applying Lemma \ref{lemma:concentration_without_replacement}. Notice that $\HH_n \subset \c{h_-,h_+}$ with $h_- = \frac{J_n^{-1/d}}{\ln^{\tau/d} n}$ and $h_+ = \frac{1}{\ln^{\tau/d} n}$, and $h_-$ satisfies $\frac{Nh_-^d}{\ln N} \gtrsim \frac{N/J_n}{\ln^{1+\tau} n} \geq \frac{n/J_n}{\ln^{1+\tau} n} \to \infty$ since in any case $J_n \ll \frac{n}{\ln^{1+\tau}n}$, so that Lemma B.12 in \cite{rosaNonparametricRegressionRandom2024a} yields $\Pro_0\p{A} \to 1$, where $A$ is the event defined by
		\[
		A := \b{\forall b_0 h^2 \ln n 
			\leq t \leq \frac{b_1}{\ln^2 n}, \quad \max_{x \in \XX_N} p_t(x,x) \leq b_2 t^{-d/2}, \quad \forall b_0 \leq J \leq b_1 \frac{h^{-d}}{\ln^{2d} N}, \quad \lambda_J \leq b_2 J^{2/d} \ln^3 n}
		\]
		for some $b_0,b_1,b_2>0$ and where
		\[
		p_t(x,x) := \sum_{j=1}^N e^{-t\lambda_j} u_j^2(x).
		\]
		In particular, for any $x \in \XX_N$ and $b_0 \leq J \leq b_1 \frac{h^{-d}}{\ln^{2d} N}$, with $\Lambda = b_2 J^{2/d} \ln^3 n$,
		\[
		Q_J(x,x) := \sum_{j=1}^J u_j^2(x)= \sum_{j=1}^N \ind{j \leq J} u_j^2(x) \leq \sum_{j=1}^N \ind{\lambda_j \leq \Lambda} u_j^2(x) \leq e p_{\Lambda^{-1}}(x,x),
		\]
		which then implies, provided that $b_0 h^2 \ln n \leq \Lambda^{-1} \leq \frac{b_1}{\ln^2 n}$,
		\[
		Q_J(x,x) \leq eb_2 \Lambda^{d/2} = eb_2^{1+d/2} J \ln^{3d/2} n.
		\]
		This is valid when $\p{b_1b_2}^{-d/2} \ln^{-d/2} n \leq J = \frac{h^{-d}}{\ln^{\tau} n} \leq \p{b_0 b_2}^{-d/2} \frac{h^{-d}}{\ln^{5d/2} n}$, which holds for large $n$ since $\tau > 5d/2$ and $\kappa > -d/2$.
		
		Applying Lemma \ref{lemma:concentration_without_replacement} we find, with $C^2 = eb_2^{1+d/2} \ln^{3d/2} n$ and for every $h \in \HH_n$ and $t>0$,
		\begin{align*}
			\Pro_{0,\sigma}\c{\norm{\Sigma_n - \Sigma_N} > t} & \leq 2N \exp \p{-\frac{1}{8} \min \p{\p{\frac{nt}{C^2J}}^2 \wedge \frac{nt}{C^2J}}} \\
			& \leq 2N \exp \p{-\frac{1}{8} \min \p{\p{\frac{nt}{C^2J_n}}^2 \wedge \frac{nt}{C^2J_n}}}.
		\end{align*}
		In particular we have $J_n \max_{h \in \HH_n}\Pro_0 \c{\norm{\Sigma_n - \Sigma_N} > t } \to 0$ whenever $J_n \ll nt/C^2 \ln n \asymp nt/\ln^{1+3d/2} n$, which is satisfied by definition of $J_n$. We also make the observation that with high $\Pro_{0,\sigma}-$probability, the matrix $\Sigma_N$ is well conditionned : indeed, for any $a \in \RR^J$, with $f = \sum_{j=1}^J a_j u_j$ we have
		\begin{align*}
			\p{\min_{x \in \XX_N} \nu_x/N} \norm{a}^2 = \p{\min_{x \in \XX_N} \nu_x/N} \norm{f}_{L^2\p{\nu}}^2 & \leq a^T \Sigma_N a = \norm{f}_N^2 = \frac{1}{N} \sum_{i=1}^N f^2(x_i) \\
			& \leq \p{\max_{x \in \XX_N} \nu_x/N} \norm{f}_{L^2\p{\nu}}^2 = \p{\max_{x \in \XX_N} \nu_x/N} \norm{a}^2,
		\end{align*}
		and moreover using Theorem B.2 in \cite{rosaNonparametricRegressionRandom2024a} we find
		\[
		\Pro_0\p{a_0 \leq \min_{x \in \XX_N} \nu_x/N \leq \max_{x \in \XX_N} \nu_x/N \leq a_1} \geq 1 - \exp \p{-cNh_-^d},
		\]
		for some $c,a_0,a_1 > 0$. As a consequence we have $\Pro_0\p{a_0 \leq \lambda_{\min}\p{\Sigma_N} \leq \lambda_{\max}\p{\Sigma_N} \leq a_1} \geq 1 - \exp \p{-cNh_-^d}$. Now if $\norm{\Sigma_n - \Sigma_N} \leq t$ we have, for any $f = \sum_{j=1}^J a_j u_j \in V_J$,
		\[
		\norm{f}_N^2 = a^T \Sigma_N a = a^T \p{\Sigma_N - \Sigma_n} a + \norm{f}_n^2 \leq \norm{\Sigma_N - \Sigma_n} \norm{a}^2 + \norm{f}_n^2 \leq \p{t \norm{\Sigma_n^{-1}} + 1} \norm{f}_n^2.
		\]
		To control $\norm{\Sigma_n^{-1}}$ we observe that, on the event where $a_0 \leq \lambda_{\min}\p{\Sigma_N} \leq \lambda_{\max}\p{\Sigma_N} \leq a_1$,
		\[
		\Sigma_n = \Sigma_N + \Sigma_n - \Sigma_N = \Sigma_N \p{I_J + \Sigma_N^{-1}\p{\Sigma_n - \Sigma_N}},
		\]
		so that, as long as $\norm{\Sigma_N^{-1}\p{\Sigma_n - \Sigma_N}} < 1$ (which is satisfied if $t < a_0$, since in that case $\norm{\Sigma_N^{-1}\p{\Sigma_n - \Sigma_N}} \leq \norm{\Sigma_N^{-1}}\norm{\p{\Sigma_n - \Sigma_N}} \leq t/a_0 < 1$),
		\begin{align*}
			\norm{\Sigma_n^{-1}} & = \norm{\p{I_J + \Sigma_N^{-1}\p{\Sigma_n - \Sigma_N}}^{-1} \Sigma_N^{-1}} \\
			& = \norm{\sum_{l \geq 0} \p{-1}^l \p{\Sigma_N^{-1}\p{\Sigma_n - \Sigma_N}}^l \Sigma_N^{-1}} \\
			& \leq \sum_{l \geq 0} \norm{\p{\Sigma_N^{-1}\p{\Sigma_n - \Sigma_N}}^l \Sigma_N^{-1}} \\
			& \leq \sum_{l \geq 0} \norm{\p{\Sigma_N^{-1}\p{\Sigma_n - \Sigma_N}}^l \Sigma_N^{-1}} \\
			& \leq \sum_{l \geq 0} \norm{\p{\Sigma_N^{-1}\p{\Sigma_n - \Sigma_N}}^l} \norm{\Sigma_N^{-1}} \\
			& \leq \sum_{l \geq 0} \norm{\Sigma_N^{-1}\p{\Sigma_n - \Sigma_N}}^l \norm{\Sigma_N^{-1}} \\
			& \leq \sum_{l \geq 0} \norm{\Sigma_N^{-1}}^l \norm{\Sigma_n - \Sigma_N}^l \norm{\Sigma_N^{-1}} \\
			& \leq \sum_{l \geq 0} \p{t/a_0}^{-l} a_1 \\
			& = \frac{a_1}{1 - t/a_0}.
		\end{align*}
		Thus, choosing $t = \frac{1}{a_1 + a_0^{-1}}$ yields, on the event where $a_0 \leq \lambda_{\min}\p{\Sigma_N} \leq \lambda_{\max}\p{\Sigma_N} \leq a_1$, for every $f \in V_J$,
		\[
		\norm{f}_N^2 \leq 2\norm{f}_n^2.
		\]
		Therefore, choosing $M_4 = 2$ we get
		\begin{align*}
			& J_n \max_{h \in \HH_n}\Pro_0 \c{\sup_{f \in V_{J_n}} \frac{\norm{f}_N}{\norm{f}_n} > M_4 } \\
			& \leq J_n \Pro_{0,\sigma}\p{a_0 > \lambda_{\min}\p{\Sigma_N} \text{ or } \lambda_{\max}\p{\Sigma_N} > a_1} + J_n \Pro_{0,\sigma}\p{\norm{\Sigma_n - \Sigma_N}>t} \\
			& \leq J_n \exp \p{-cNh_-^d} + 2N^2 \exp \p{- \frac{1}{8} \min \p{\p{\frac{nt}{C^2J_n}}^2 \wedge \frac{nt}{C^2J_n}}} \\
			& = o(1),
		\end{align*}
		which concludes the proof.
	\end{proof}
	
	\begin{lemma}\label{lemma:concentration_without_replacement}
		Let $\YY_n$ be a set of cardinality $N \geq 1$, $\p{u_j}_{j=1}^N$ be a family of functions over $\YY_N$ such that,  for some $J_0,J_1,C>0$, for every $J_0 \leq J \leq J_1$ and $y \in \YY_N$,
		\[
		Q_J\p{y,y} := \frac{1}{J} \sum_{j=1}^J u_j(y)^2 \leq C^2.
		\]
		Then if $n \leq N$ and $\p{X_i}_{i=1}^n$ are sampled uniformly without replacement from $\YY_N$, $J_0 \leq J \leq J_1$ and $\Sigma_n, \Sigma_N \in \RR^{J \times J}$ are defined by
		\[
		\p{\Sigma_n}_{jk} := \frac{1}{n} \sum_{i=1}^n u_j(x_i) u_k(x_i), \quad \p{\Sigma_N}_{jk} := \frac{1}{N} \sum_{y \in \YY} u_j(y) u_k(y),
		\]
		we have
		\[
		\Pro \c{\norm{\Sigma_n - \Sigma_N} > t} \leq 2N \exp \p{- \frac{1}{4} \min \p{ \p{\frac{nt}{C^2 J}}^2 \wedge \frac{nt}{C^2J}}}.
		\]
	\end{lemma}
	\begin{proof}
		The proof is similar to the one of Theorem \ref{theorem:1}, adapted to the setting of sampling without replacement. We have
		\[
		\Sigma_n = \frac{1}{n} \sum_{i=1}^n K(x_i), \quad \Sigma_N = \frac{1}{N} \sum_{y \in \YY_N} K(y),
		\]
		where, for any $y \in \YY_N$,
		\[
		K(y) := uu^T(y), \quad u(y) = \c{u_1(y),\ldots,u_J(y)}^T.
		\]
		In this context, Theorem 1 in \cite{grossNoteSamplingReplacing2010} applied to the summands $\p{\p{K(x_i) - \Sigma_N}/n}_{i=1}^n$ and
		\[
		S = \sum_{i=1}^n \frac{K(x_i) - \Sigma_N}{n} = \Sigma_n - \Sigma_N
		\]
		implies that
		\[
		\Pro \c{\norm{\Sigma_n - \Sigma_N} > t} = \Pro \c{\norm{S} > t} \leq 2N \exp \p{- \min \p{\frac{t^2}{4V} \wedge \frac{t}{2c}}},
		\]
		where
		\[
		c = \max_{y \in \YY_N} \norm{\frac{K(y) - \Sigma_N}{n}} \leq \frac{2}{n} \max_{y \in \YY_N} \norm{K(y)} = \frac{2}{n} \max_{y \in \YY_N} \norm{u(y)}_{\RR^J}^2 \leq \frac{C^2 J}{n},
		\]
		and
		\begin{align*}
			V & = \norm{\frac{1}{N} \sum_{y \in \YY_N} \p{\frac{K(y) - \Sigma_N}{n}}^2} \\
			& = \norm{\frac{1}{N} \sum_{y \in \YY_N} \frac{K(y)^2 + \Sigma_N^2 - K(y) \Sigma_N - \Sigma_N K(y)}{n^2}} \\
			& = \norm{\frac{1}{Nn^2} \sum_{y \in \YY_N} K(y)^2 - \Sigma_N^2} \\
			& \leq \frac{1}{n^2} \p{\norm{\Sigma_N^2} + \max_{y \in \YY_N} \norm{K(y)^2}} \\
			& \leq \frac{1}{n^2} \p{\norm{\Sigma_N}^2 + \max_{y \in \YY_N} \norm{K(y)}^2} \\
			& = \frac{1}{n^2} \p{\norm{\Sigma_N}^2 + \max_{y \in \YY_N} \norm{u(y)}_{\RR^J}^4}.
		\end{align*}
		Moreover $\norm{u(y)}_{\RR^J}^2 \leq C^2 J$ and
		\[
		\norm{\Sigma_N} = \norm{\frac{1}{N} \sum_{y \in \YY_N} K(y)} \leq \max_{y \in \YY_N} \norm{K(y)} = \max_{y \in \YY_N} \norm{u(y)}_{\RR^J}^2 \leq C^2J,
		\]
		which yields $V \leq C^4J^2/n^2$ and therefore
		\begin{align*}
		\Pro \c{\norm{\Sigma_n - \Sigma_N} > t} & \leq 2N \exp \p{- \min \p{\frac{t^2}{4V} \wedge \frac{t}{2c}}} \\
		& \leq 2N \exp \p{- \min \p{\frac{n^2 t^2}{4C^4J^2} \wedge \frac{nt}{2C^2J}}} \\
		& \leq 2N \exp \p{- \frac{1}{4} \min \p{ \p{\frac{nt}{C^2 J}}^2 \wedge \frac{nt}{C^2J}}}.
		\end{align*}
	\end{proof}
	
\end{document}